\documentclass[a4paper,12pt]{amsart}
\usepackage{amsfonts,amsmath,amssymb,latexsym,hyperref,color,tikz}

\allowdisplaybreaks
\numberwithin{equation}{section}
\newtheorem{theorem}{Theorem}[section]
\newtheorem{proposition}[theorem]{Proposition}
\newtheorem{lemma}[theorem]{Lemma}
\newtheorem{corollary}[theorem]{Corollary}

\newtheorem{claim}[theorem]{Claim}
\theoremstyle{definition}
\newtheorem{definition}[theorem]{Definition}

\theoremstyle{remark}
\newtheorem{remark}[theorem]{Remark}
\newtheorem{example}[theorem]{Example}
\newtheorem*{conv}{Convention}

\begin{document}
\title{Efficient presentations of relative Cuntz-Krieger algebras}

\author{Lisa Orloff Clark}
\address{Lisa Orloff Clark\\School of Mathematics and Statistics\\Victoria University of Wellington 
\\PO Box 600
\\Wellington 6140\\New Zealand}

%
 \author{Yosafat E. P. Pangalela}
\address{Yosafat E.P. Pangalela\\
Pinnacle Investment\\Wisma GKBI 38th Floor, Suite 3805 \\ Jl Jendral Sudirman No 28 \\
Jakarta 10210 \\ Indonesia}
\email{yosafat.pangalela@hotmail.com}


\begin{abstract}
In this article, we present a new method to study relative Cuntz-Krieger
algebras for higher-rank graphs. We only work with edges rather than paths
of arbitrary degrees. We then use this method to simplify the existing
results about relative Cuntz-Krieger algebras. We also give applications to
study ideals and quotients of Toeplitz algebras.
\end{abstract}

\thanks{This research was supported by 
Marsden grant 15-UOO-071 from the Royal Society of New Zealand.  \\
Thank you to Iain Raeburn for sharing his insights.  }

\maketitle

\section{Introduction}

\label{Section-Introduction}

For a directed graph $E$, Fowler and Raeburn introduced the Toeplitz algebra
$TC^{\ast }(E)$ \cite{FR99}. The usual graph algebra $C^{\ast }(E)$ (or
\emph{Cuntz-Krieger algebra}) is the quotient of $TC^{\ast }(E)$ in which
the \emph{Cuntz-Krieger relation}%
\begin{equation*}
p_{v}=\sum_{r(e)=v}s_{e}s_{e}^{\ast }
\end{equation*}%
is imposed at every \emph{regular vertex} $v$; that is, at every vertex that
receives only finitely many edges. Muhly and Tomforde described the
quotients of $C^{\ast }(E)$ as \emph{relative graph algebras}: for a set $V$
of regular vertices, the relative graph algebra $C^{\ast }(E;V)$ is the
quotient of $TC^{\ast }(E)$ in which the Cuntz-Krieger relation is imposed
at every $v\in V$ \cite{MT04}. These relations are independent of each
other: if $v\notin V$, then $p_{v}\neq \sum_{e\in vE^{1}}s_{e}s_{e}^{\ast }$
in $C^{\ast }(E;V)$.

The higher-rank graphs or \emph{$k$-graphs} of Kumjian and Pask \cite{KP00}
are higher-dimensional analogues of directed graphs, and they also have both
a Cuntz-Krieger algebra \cite{KP00,RSY03, RSY04} and a Toeplitz algebra \cite%
{RS05}. Here we consider the class of \emph{finitely aligned} $k$-graphs
discovered in \cite{RS05} and studied in \cite{RSY04}. For such a $k$-graph $%
\Lambda $, the Cuntz-Krieger algebra $C^{\ast }(\Lambda )$ is the quotient
of the Toeplitz algebra $TC^{\ast }(\Lambda )$ in which the \emph{%
Cuntz-Krieger relation}%
\begin{equation}
\prod_{\lambda \in E}( t_{v}-t_{\lambda }t_{\lambda }^{\ast }) =0
\label{CKE}
\end{equation}%
is imposed at every \emph{finite exhaustive set }$E$ of every $v\Lambda
:=r^{-1}(v)$. However, the relations \eqref{CKE} are not independent of each
other: imposing the relations for some exhaustive sets automatically imposes
others. For a collection $\mathcal{E}$\ of finite exhaustive sets of $%
\Lambda $, Sims introduced the \emph{relative Cuntz-Krieger algebras} $%
C^{\ast }(\Lambda ;\mathcal{E})$ to be the quotient of $TC^{\ast }(\Lambda )$
in which the relation \eqref{CKE} is imposed for every $E\in \mathcal{E}$.
He also identified \emph{satiated} collections $\mathcal{E}$ of finite
exhaustive sets which describe the possible quotients $C^{\ast }(\Lambda ;%
\mathcal{E})$ \cite{Si06}.

Sims's satiated sets are huge, and his exhaustive sets can include paths of
arbitrary degrees. However, while it has been standard since the beginning
of the subject \cite{KP00} to work with Cuntz-Krieger relations of all
degrees, we know from \cite[Appendix~C]{RSY04} that it is possible to work
only with sets of edges, as one does for directed graphs (the $1$-graphs),
and that it is then easier to see what is happening. Here we describe and
study a family of \emph{efficient} collections in which the exhaustive sets
contain only a minimal number of edges. This simplifies the description of
the relations being imposed when passing from the Toeplitz algebra to a
relative Cuntz-Krieger algebra.

This paper is organised as follows. In Section \ref{Section-k-graph}, we
give the definition of higher-rank graphs and establish our notation. In
Section \ref{Section-efficient}, we introduce efficient sets and give
examples. For a directed graph $E$, every set of regular vertices of $E$ 
can be viewed as
an efficient set (Example \ref{example-efficient-graph}). Hence relative
graph algebras for directed graphs \cite{MT04} were defined using efficient
sets rather than Sims's satiated sets.

In\ Section \ref{Section-boundary-paths}, we introduce $\mathcal{E}$%
-boundary paths (Definition \ref{compatible-boundary-paths}). We use these
paths to establish properties of the universal relative Cuntz-Krieger
algebras (Proposition \ref{properties-of-relative-CK-algebra}).\ In Section %
\ref{Section-efficient-and-satiated}, we discuss the satiated sets of Sims
\cite{Si06} and show that efficient sets are in bijective correspondence
with satiated sets (Theorem \ref{bijection-efficient-and-satiated}).

Finally, we discuss applications in\ Section \ref{Section-applications}. The
first application is a new version of the gauge-invariant uniqueness theorem
for relative Cuntz-Krieger algebras of \cite{Si06} (Theorem \ref%
{the-gauge-invariant-uniqueness-theorem}). The second application is to
simplify a complete listing of the gauge-invariant ideals in a relative
Cuntz-Krieger algebra of \cite{SWW14} (Theorem \ref%
{gauge-invariant-ideals-efficient}). The authors of \cite{Si06} and \cite%
{SWW14} use satiated sets to formulate both results. By translating these
into efficient sets, we provide alternative versions, which are considerably
more checkable. In the last application, we investigate the relationship
among $k$-graphs Toeplitz algebras and their ideals and quotients
(Proposition \ref{diagram-Toeplitz-algebra}).

\section{Preliminaries}

\label{Section-k-graph}

Let $k$ be a positive integer. We regard $\mathbb{N}^{k}$ as an additive
semigroup with identity $0$. We write $n\in \mathbb{N}^{k}$ as $(
n_{1},\ldots ,n_{k}) $ and define $\left\vert n\right\vert :=\sum_{1\leq
i\leq k}n_{i}$. We denote the usual basis of $\mathbb{N}^{k}$\ by $\{
e_{i}\} $. For $m,n\in \mathbb{N}^{k}$, we write $m\leq n$ to denote $%
m_{i}\leq n_{i}$ for $1\leq i\leq k$. We also write $m\vee n$ for\ their
coordinate-wise maximum and $m\wedge n$ for their coordinate-wise minimum.

A \emph{higher-rank graph} or $k$\emph{-graph} is a countable category $%
\Lambda $ endowed with a functor $d:\Lambda \rightarrow \mathbb{N}^{k}$
satisfying the \emph{factorisation property}: for $\lambda \in \Lambda $ and
$m,n\in \mathbb{N}^{k}$ with $d( \lambda ) =m+n$, there are unique elements $%
\mu ,\nu \in \Lambda $ such that $\lambda =\mu \upsilon $, $d( \mu ) =m$ and
$d( \nu ) =n$. We then write $\lambda ( 0,m) $ for $\mu $ and $\lambda (
m,m+n) $ for $\nu $.

For $n\in \mathbb{N}^{k}$, we define%
\begin{equation*}
\Lambda ^{n}:=\{ \lambda \in \Lambda :d(\lambda )=n\}
\end{equation*}%
and call the elements $\lambda $ of $\Lambda ^{n}$ \emph{paths of degree }$n$%
. For $1\leq i\leq k$, a path $e\in \Lambda ^{e_{i}}$ is an \emph{edge}, and
we write
\begin{equation*}
\Lambda ^{1}:=\bigcup_{1\leq i\leq k}\Lambda ^{e_{i}}
\end{equation*}%
for the set of all edges. We regard elements of $\Lambda ^{0}$ as \emph{%
vertices}. For $v\in \Lambda ^{0}$, $\lambda \in \Lambda $ and $E\subseteq
\Lambda $, we define
\begin{align*}
vE& :=\{ \mu \in E:r(\mu )=v\} \text{, } \\
\lambda E& :=\{ \lambda \mu \in \Lambda :\mu \in E,r(\mu )=s(\lambda )\}
\text{,} \\
E\lambda & :=\{ \mu \lambda \in \Lambda :\mu \in E,s(\mu )=r(\lambda )\}
\text{.}
\end{align*}%
A $k$-graph $\Lambda $ is \emph{row-finite} if for $v\in \Lambda ^{0}$ and $%
1\leq i\leq k$, the set $v\Lambda ^{e_{i}}$ is finite. A vertex $v\in
\Lambda ^{0}$ is a \emph{source} if there exists $m\in \mathbb{N}^{k}$ such
that $v\Lambda ^{m}=\emptyset $.

To visualise $k$-graphs, we use coloured graphs of \cite{HRSW13}. For a $k$%
-graph $\Lambda $, we choose $k$-different colours $c_{1},\ldots ,c_{k}$ and
associate each edge $e\in \Lambda ^{e_{i}}$ to an edge of colour $c_{i}$. We
call this coloured graph the \emph{skeleton }of $\Lambda $.

\begin{conv}
We draw
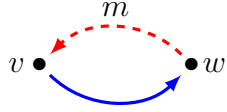
\begin{equation*}
\begin{tikzpicture} \node[inner sep=1pt] (v) at (0,0) {$\bullet$};
\node[inner sep=1pt] at (-0.3,0) {$v$}; \node[inner sep=1pt] (w) at (2,0)
{$\bullet$}; \node[inner sep=1pt] at (2.3,0) {$w$}; \draw[-latex, blue, very
thick] (v) edge [out=-45, in=-135](w); \draw[-latex, red, dashed, very
thick] (w) edge [out=135, in=45] node[pos=0.5, above, black]{$m$}(v);
\end{tikzpicture}
\end{equation*}%
in the skeleton of a $2$-graph to denote that there are $m$ $( 1,0) $-edges
from $w$ to $v$ and $1$ $( 0,1) $-edge from $v $ to $w$.
\end{conv}

For $\lambda ,\mu \in \Lambda $, we define
\begin{align*}
\operatorname{MCE}(\lambda ,\mu ):=& \{\tau \in \Lambda ^{d(\lambda )\vee d(\mu
)}:\tau (0,d(\lambda ))=\lambda \text{ and }\tau (0,d(\mu ))=\mu \}\text{,}
\\
\Lambda ^{\min }(\lambda ,\mu )& :=\{ (\lambda ^{\prime },\mu ^{\prime })\in
\Lambda \times \Lambda :\lambda \lambda ^{\prime }=\mu \mu ^{\prime }\in
\operatorname{MCE}(\lambda ,\mu )\} \text{.}
\end{align*}%
We say that $\Lambda $ is \emph{finitely aligned} if $\Lambda ^{\min
}(\lambda ,\mu )$ is finite (possibly empty) for all $\lambda ,\mu \in
\Lambda $.

For $v\in \Lambda ^{0}$, $E\subseteq v\Lambda $ is \emph{exhaustive} if for $%
\lambda \in v\Lambda $, there exists $\mu \in E$ with $\Lambda ^{\min }(
\lambda ,\mu ) \neq \emptyset $. We write $\operatorname{FE}( \Lambda ) $ to denote
the collection of finite exhaustive sets $E\subseteq v\Lambda \backslash \{
v\} $, and%
\begin{equation*}
\operatorname{FE}( \Lambda ^{1}) :=\operatorname{FE}( \Lambda ) \cap \Lambda ^{1}\text{.}
\end{equation*}%
For $E\in \operatorname{FE}( \Lambda ) $, we write $r( E) $ for the vertex $v\in
\Lambda ^{0}$ such that $E\subseteq v\Lambda $.

A \emph{Toeplitz-Cuntz-Krieger }$\Lambda $\emph{-family} is a collection $\{
T_{\lambda }:\lambda \in \Lambda \} $ of partial isometries in a $C^{\ast }$%
-algebra $B$ satisfying:

\begin{enumerate}
\item[(TCK1)] $\{ T_{v}:v\in \Lambda ^{0}\} $ is a collection of mutually
orthogonal projections;

\item[(TCK2)] $T_{\lambda }T_{\mu }=T_{\lambda \mu }$ whenever $s( \lambda )
=r( \mu ) $; and

\item[(TCK3)] $T_{\lambda }^{\ast }T_{\mu }=\sum_{(\lambda ^{\prime },\mu
^{\prime })\in \Lambda ^{\min }( \lambda ,\mu ) }T_{\lambda ^{\prime
}}T_{\mu ^{\prime }}^{\ast }$ for all $\lambda ,\mu \in \Lambda $.
\end{enumerate}

For $\mathcal{E}\subseteq \operatorname{FE}( \Lambda ) $, a \emph{relative
Cuntz-Krieger }$( \Lambda ;\mathcal{E}) $\emph{-family} is a
Toeplitz-Cuntz-Krieger $\Lambda $-family $\{ T_{\lambda }:\lambda \in
\Lambda \} $ which satisfies the \emph{Cuntz-Krieger relations}:

\begin{enumerate}
\item[(CK)] $\prod_{\lambda \in E}( T_{r( E) }-T_{\lambda }T_{\lambda
}^{\ast }) =0$ for all $E\in \mathcal{E}$.
\end{enumerate}

In \cite[Section 4]{RS05}, Raeburn and Sims proved that there is a $C^{\ast
} $-algebra $TC^{\ast }(\Lambda )$ generated by a universal
Toeplitz-Cuntz-Krieger $\Lambda $-family $\{ t_{\lambda }:\lambda \in
\Lambda \} $. For $\mathcal{E}\subseteq \operatorname{FE}(\Lambda )$, the quotient $%
C^{\ast }(\Lambda ;\mathcal{E})$ of $TC^{\ast }(\Lambda )$ by the ideal
generated by%
\begin{equation*}
\Big\{\prod_{\lambda \in E}(t_{r(E)}-t_{\lambda }t_{\lambda }^{\ast }):E\in
\mathcal{E}\Big\}
\end{equation*}%
is generated by a universal relative Cuntz-Krieger $(\Lambda ;\mathcal{E})$%
-family $\{s_{\lambda }^{\mathcal{E}}:\lambda \in \Lambda \}$. For a
relative Cuntz-Krieger $(\Lambda ;\mathcal{E})$-family $\{ S_{\lambda
}:\lambda \in \Lambda \} $ in $C^{\ast }$-algebra $B$, we write $\pi _{S}$
for the homomorphism of $C^{\ast }(\Lambda ;\mathcal{E})$ into $B$ such that $%
\pi _{S}(s_{\lambda }^{\mathcal{E}})=S_{\lambda }$ for $\lambda \in \Lambda $%
.

\section{Efficient sets}

\label{Section-efficient}

In this section, we introduce efficient sets and study their properties.
Imposing the Cuntz-Krieger relations on a higher-rank graph has domino
effects, which are described in Proposition \ref{efficient-exhaustive-edge}
and Proposition \ref{properties-of-ideal}, and these effects motivate
Definition \ref{definition-efficient}.

\begin{proposition}
\label{efficient-exhaustive-edge} Let $v\in \Lambda ^{0}$, $E\in v\operatorname{FE%
}( \Lambda ^{1}) $ and $f\in r( E) \Lambda ^{1}\backslash E$. Suppose that $%
\{ T_{\lambda }:\lambda \in \Lambda \} $ is a Toeplitz-Cuntz-Krieger $%
\Lambda $-family such that $\prod_{e\in E}(T_{v}-T_{e}T_{e}^{\ast })=0$.
Define
\begin{equation*}
\operatorname{Ext}_{\Lambda }( f;E) :=\{g\in s( f) \Lambda :fg\in \operatorname{MCE}( f,e)
\text{ for some }e\in E\}\text{.}
\end{equation*}%
Then $\operatorname{Ext}_{\Lambda }(f;E)\in s( f) \operatorname{FE}( \Lambda ^{1}) $ and $%
\prod_{g\in \operatorname{Ext}_{\Lambda }(f;E)}(T_{s( f) }-T_{g}T_{g}^{\ast })=0.$
\end{proposition}

\begin{proof}
Since $E\in \operatorname{FE}( \Lambda ) $, \cite[Lemma C.5]{RSY04} gives 
$\operatorname{Ext}_{\Lambda }(f;E)\in s( f) \operatorname{FE}( \Lambda ) $.   
We claim that $\operatorname{Ext}_{\Lambda }(f;E)\subseteq s( f) \Lambda ^{1}$.  To prove the claim, let $g\in\operatorname{Ext}_{\Lambda }(f;E)$.  Then there exists an edge $e\in E$ such that $fg\in \operatorname{MCE}( f,e) $.   First we show that $d(e)\neq d(f)$.   
Suppose for a contradiction that $d( f) =d( e) $. Since $fg\in\operatorname{MCE}( f,e) $,
\begin{equation*}
d( f) +d( g) =d( fg) =d( f) \vee d( e) =d( f) \text{.}
\end{equation*}%
Hence $d( g) =0$ and $g=s( f) $. Because $fg\in \operatorname{MCE}( f,e) $, $f=e\in
E $, which contradicts $f\notin E$. Thus $d( f) \neq d( e) $.
Now since $\left\vert d(f)\right\vert =1=\left\vert d(e)\right\vert $, $d( f) \vee d(
e) =d( f) +d( e) $. So
\begin{equation*}
d( f) +d( g) =d( fg) =d( f) \vee d( e) =d( f) +d( e) 
\end{equation*}
and hence $d( g) =d( e) $ and $\left\vert d(g)\right\vert =1 $ proving the claim.
Thus $\operatorname{Ext}_{\Lambda}(f;E)\in s( f) \operatorname{FE}( \Lambda ^{1}) $.

To show $\prod_{g\in \operatorname{Ext}_{\Lambda }(f;E)}(T_{s( f) }-T_{g}T_{g}^{\ast
})=0$, we prove
\begin{equation*}
\prod_{g\in \operatorname{Ext}_{\Lambda }(f;E)}(T_{s( f) }-T_{g}T_{g}^{\ast
})=\prod_{e\in E}(T_{f}^{\ast }( T_{v}-T_{e}T_{e}^{\ast }) T_{f})=0\text{.}
\end{equation*}%
First by (TCK3), we have
\begin{equation*}
\prod_{e\in E}(T_{f}^{\ast }( T_{v}-T_{e}T_{e}^{\ast }) T_{f})=\prod_{e\in E}%
\Big(T_{s( f) }-\Big(\sum_{(g,e^{\prime }),(g^{\prime \prime },e^{\prime
\prime })\in \Lambda ^{\min }( f,e) }T_{g}( T_{e^{\prime }}^{\ast
}T_{e^{\prime \prime }}) T_{g^{\prime \prime }}^{\ast }\Big)\Big)\text{.}
\end{equation*}%
For $e^{\prime }\neq e^{\prime \prime }$, $d( e^{\prime }) =d( e^{\prime
\prime }) $ and by (TCK3), $T_{e^{\prime }}^{\ast }T_{e^{\prime \prime }}=0$%
. If $e^{\prime }=e^{\prime \prime }$, then $T_{e^{\prime }}^{\ast
}T_{e^{\prime }}=T_{s(e^{\prime })}$. So%
\begin{align*}
\prod_{e\in E}(T_{f}^{\ast }( T_{v}-T_{e}T_{e}^{\ast }) T_{f})& =\prod_{e\in
E}\Big(T_{s( f) }-\Big(\sum_{(g,e^{\prime })\in \Lambda ^{\min }( f,e)
}T_{g}T_{g}^{\ast }\Big)\Big) \\
& =\prod_{e\in E}\prod_{(g,e^{\prime })\in \Lambda ^{\min }( f,e) }(T_{s( f)
}-T_{g}T_{g}^{\ast })=\prod_{g\in \operatorname{Ext}_{\Lambda }(f;E)}(T_{s( f)
}-T_{g}T_{g}^{\ast })\text{.}
\end{align*}%
On the other hand, since $\{ T_{\lambda }T_{\lambda }^{\ast }:\lambda \in
\Lambda \} $ is a commuting family \cite[Lemma 2.7(i)]{RSY04},%
\begin{equation*}
\prod_{e\in E}(T_{f}^{\ast }( T_{v}-T_{e}T_{e}^{\ast }) T_{f})=T_{f}^{\ast }%
\Big(\prod_{e\in E}(T_{v}-T_{e}T_{e}^{\ast })\Big)(T_{f}T_{f}^{\ast
})^{\left\vert E\right\vert -1}T_{f}=0\text{.}
\end{equation*}
\end{proof}

\begin{corollary}
\label{efficient-exhaustive}
Let $v\in \Lambda ^{0}$ and 
$E\in v\operatorname{FE}( \Lambda ^{1}) $. For $\lambda \in \left. r( E) \Lambda
\right\backslash E\Lambda $, the set
\begin{equation*}
\operatorname{Ext}_{\Lambda }( \lambda ;E) :=\{g\in s( \lambda ) \Lambda :\lambda
g\in \operatorname{MCE}( \lambda ,e) \text{ for some }e\in E\}
\end{equation*}%
belongs to $s( \lambda ) \operatorname{FE}( \Lambda ^{1}) $.
\end{corollary}

\begin{proof}
For $w\in \Lambda ^{0}$, $F\subseteq w\Lambda $, $\lambda _{1}\in w\Lambda $
and $\lambda _{2}\in s(\lambda _{1})\Lambda $, Lemma 4.9 of \cite{Si06}
tells that%
\begin{equation*}
\operatorname{Ext}_{\Lambda }(\lambda _{1}\lambda _{2};F)=\operatorname{Ext}_{\Lambda
}(\lambda _{2};\operatorname{Ext}_{\Lambda }(\lambda _{1};F))\text{.}
\end{equation*}%
Hence by induction on $\left\vert \lambda \right\vert $, $\operatorname{Ext}%
_{\Lambda }( \lambda ;E) \in \operatorname{FE}( \Lambda ^{1}) $ follows from
Proposition \ref{efficient-exhaustive-edge}.
\end{proof}

\begin{proposition}
\label{properties-of-ideal}Let $v\in \Lambda ^{0}$ and $E\in v\operatorname{FE}(
\Lambda ^{1}) $. Suppose that $\{ T_{\lambda }:\lambda \in \Lambda \} $ is a
Toeplitz-Cuntz-Krieger $\Lambda $-family such that $\prod_{g\in
E}(T_{v}-T_{g}T_{g}^{\ast })=0$.  Also suppose that $e\in E$ and  $F\in s( e)
\operatorname{FE}( \Lambda ^{1}) $ satisfies $\prod_{f\in F}(T_{s( e)
}-T_{f}T_{f}^{\ast })=0$. 
Then
\begin{equation*}
E_{F}:=( E\backslash \{ e\} ) \cup \{ ( ef) ( 0,d( f) ) :f\in F\} 
\end{equation*}
belongs to $v\operatorname{FE}( \Lambda ^{1}) $ and $\prod_{g\in
E_{F}}(T_{v}-T_{g}T_{g}^{\ast })=0$.
\end{proposition}

\begin{proof}
Since $E$ and $F$ are nonempty and finite, 
so is $E_{F}$. To show that $E_{F}$ is exhaustive, take
$\lambda \in v\Lambda $. We give a separate argument for $\Lambda ^{\min
}(\lambda ,e)=\emptyset $ and $\Lambda ^{\min }(\lambda ,e)\neq \emptyset $.
First suppose $\Lambda ^{\min }(\lambda ,e)=\emptyset $. Because $E\in v%
\operatorname{FE}(\Lambda ^{1})$ and $\Lambda ^{\min }(\lambda ,e)=\emptyset $,
there exists $g\in E\backslash \{e\}$ with $\Lambda ^{\min }(\lambda ,g)\neq
\emptyset $. By definition of $E_{F}$, $g\in E_{F}$ as needed. 

Next suppose $\Lambda ^{\min }(\lambda ,e)\neq \emptyset $. Take 
$(\lambda ^{\prime },e^{\prime })\in \Lambda ^{\min }(\lambda ,e)$. 
So $\lambda \lambda ^{\prime }=ee^{\prime }$. Since $F\in r(e^{\prime })\operatorname{FE}%
(\Lambda ^{1})$, there exists $f\in F$ with $\Lambda ^{\min }(e^{\prime
},f)\neq \emptyset $. Take $(e^{\prime \prime },f^{\prime \prime })\in
\Lambda ^{\min }(e^{\prime },f)$. Then $e^{\prime }e^{\prime \prime}=ff^{\prime \prime }$,
$\lambda \lambda ^{\prime }e^{\prime \prime
}=ee^{\prime }e^{\prime \prime }=eff^{\prime \prime }$, and $\Lambda ^{\min
}(\lambda ,ef)\neq \emptyset $. Let $g:=(ef)(0,d(f))$. Then $g\in E_{F}$ and
$\Lambda ^{\min }(\lambda ,g)\neq \emptyset $. So $E_{F}$ is exhaustive and $%
E_{F}\in \operatorname{FE}(\Lambda ^{1})$.

Since $\prod_{f\in F}(T_{s( e) }-T_{f}T_{f}^{\ast })=0$, 
we have $T_{v}-T_{e}T_{e}^{\ast }=\prod_{f\in F}(T_{v}-T_{ef}T_{ef}^{\ast })$
by \cite[Lemma C.7]{RSY04}.
Hence
\begin{equation*}
\prod_{g\in E\backslash \{ e\} }(T_{v}-T_{g}T_{g}^{\ast })\prod_{f\in
F}(T_{v}-T_{ef}T_{ef}^{\ast })=\prod_{g\in E}(T_{v}-T_{g}T_{g}^{\ast })=0%
\text{.}
\end{equation*}%
On the other hand, for $f\in F$, $T_{ef}T_{ef}^{\ast }=(T_{ef}T_{ef}^{\ast
})(T_{( ef) ( 0,d( f) ) }T_{( ef) ( 0,d( f) ) }^{\ast })$. So by (TCK3),
\begin{equation*}
\prod_{g\in E_{F}}(T_{v}-T_{g}T_{g}^{\ast })=\prod_{g\in E\backslash \{ e\}
}(T_{v}-T_{g}T_{g}^{\ast })\prod_{f\in F}((T_{v}-T_{ef}T_{ef}^{\ast
})(T_{v}-T_{( ef) ( 0,d( f) ) }T_{( ef) ( 0,d( f) ) }^{\ast }))=0\text{.}
\end{equation*}
\end{proof}

\begin{definition}
\label{definition-efficient}We call a subset $\mathcal{E}$ of $\operatorname{FE}(
\Lambda ^{1}) $ \emph{efficient} if the following three conditions are satisfied:

\begin{enumerate}
\item[(E1)] if $E,F\in \mathcal{E}$ and $E\subseteq F$, then $E=F$;

\item[(E2)] if $E\in \mathcal{E}$ and $f\in r( E) \Lambda ^{1}\backslash E$,
then there exists $F\in \mathcal{E}$ such that $F\subseteq $ $\operatorname{Ext}%
_{\Lambda }(f;E)$; and

\item[(E3)] if $E\in \mathcal{E}, e\in E$, and $F\in s( e) \mathcal{E}$,
then there exists $G\in \mathcal{E}$ with
\begin{equation*}
G\subseteq ( E\backslash \{ e\} ) \cup \{ ( ef) ( 0,d( f) ) :f\in F\} .
\end{equation*}
\end{enumerate}
\end{definition}

\begin{example}
\label{example-efficient-graph}Suppose that $\Lambda $ is a $1$-graph. Let $%
V\subseteq \Lambda ^{0}$ be a nonempty subset of regular vertices; that is, $%
0<\left\vert v\Lambda ^{1}\right\vert <\infty $ for $v\in V$. Then the set $%
\mathcal{E}:=\{v\Lambda ^{1}:v\in V\}$ is efficient:  
Properties (E1) and (E2) are trivially true.  To show (E3), let $E\in \mathcal{E}$, $e\in E$ and
  $F\in s( e) \mathcal{E} $.  (Since $s(e)$ is a  regular vertex, $s(e)\mathcal{E}$ is nonempty.)
Then \begin{equation*}
( E\backslash \{ e\} ) \cup \{ ( ef) ( 0,d( f) ) :f\in F\} =( E\backslash \{
e\} ) \cup \{ e\} =E.
\end{equation*}%
Choose $G:=E$ and (E3) follows.
So one can translate the results about subsets of regular vertices of \cite{MT04} into
results about efficient sets.
\end{example}

\begin{example}
\label{example-efficient-2-graph}Suppose that $\Lambda $ is a $2$-graph with
the following skeleton:%
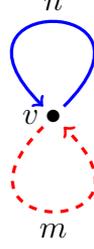
\begin{equation*}
\begin{tikzpicture} \node[inner sep=1pt] (v) at (0,0) {$\bullet$};
\node[inner sep=1pt] at (-0.3,0) {$v$}; \path[->,every
loop/.style={looseness=20}] (v) edge[in=-45,out=-135,loop, red, dashed, very
thick] node[pos=0.5, below, black]{$m$} (v); \path[->,every
loop/.style={looseness=20}] (v) edge[in=135,out=45,loop, blue, very thick]
node[pos=0.5, above, black]{$n$} (v);\end{tikzpicture}
\end{equation*}
We show $\{ v\Lambda ^{e_{1}}\} $ is efficient.  Condition  (E1) is trivial.
 Notice that for any $g\in v\Lambda ^{e_{2}}$,  we have $\operatorname{Ext}_{\Lambda
}(g;v\Lambda ^{e_{1}})=v\Lambda ^{e_{1}}$ so (E2) holds. For (E3),  
let $e\in v\Lambda ^{e_{1}}$.  Then
\begin{equation*}
( v\Lambda ^{e_{1}}\backslash \{ e\} ) \cup \{ ( ef) ( 0,d( f) ) :f\in
v\Lambda ^{e_{1}}\} =( v\Lambda ^{e_{1}}\backslash \{ e\} ) \cup \{ e\}
=v\Lambda ^{e_{1}}
\end{equation*}%
since all edges of $v\Lambda ^{e_{1}}$ have the same degree. 
A similar argument shows $\{v\Lambda ^{e_{2}}\}$ and $\{v\Lambda
^{e_{1}}\cup v\Lambda ^{e_{2}}\}$ are also efficient.
\end{example}

\begin{remark}
\label{example-efficient} For a row-finite $k$-graph $\Lambda $ with no sources and a nonempty subset $K$ of $\{1,\ldots ,k\}$,
\begin{equation*}
\mathcal{E}_{K}:=\Big\{\bigcup_{i\in K}v\Lambda ^{e_{i}}:v\in \Lambda ^{0}%
\Big\}\text{ and }\mathcal{E}_{i}:=\{v\Lambda
^{e_{i}}:v\in \Lambda ^{0}\} \text{ for } 1\leq i \leq k
\end{equation*}%
are all efficient.  So, for example,  the sets $\{ v\Lambda ^{e_{1}}:v\in \Lambda ^{0}\} $, $\{ v\Lambda
^{e_{2}}:v\in \Lambda ^{0}\} $ and $\{ v\Lambda ^{e_{1}}\cup v\Lambda
^{e_{2}}:v\in \Lambda ^{0}\} $ are always efficient.  However
 the set $\{ v\Lambda
^{e_{1}},v\Lambda ^{e_{2}}:v\in \Lambda ^{0}\} $ might not be.
Consider the $2$-graph $\Lambda $ with skeleton%
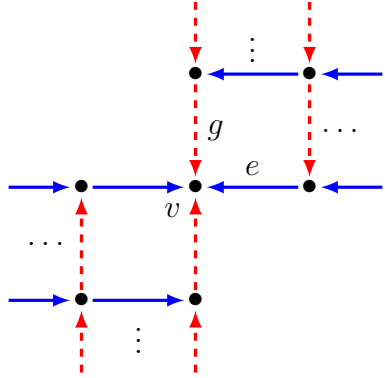
\begin{equation*}
\begin{tikzpicture} \node[inner sep=1pt] (v) at (0,0) {$\bullet$};
\node[inner sep=1pt] at (-0.3,-0.3) {$v$}; \node[inner sep=1pt] (a) at
(1.5,0) {$\bullet$};  \node[inner
sep=1pt] (b) at (0,1.5) {$\bullet$};  \node[inner sep=1pt] (c) at (1.5,1.5) {$\bullet$}; \node[inner
sep=1pt] (d) at (-1.5,0) {$\bullet$}; \node[inner sep=1pt] (e) at (0,-1.5)
{$\bullet$}; \node[inner sep=1pt] (f) at (-1.5,-1.5) {$\bullet$};
\draw[-latex, blue, very thick] (a) edge[out=180, in=0] node[pos=0.5, above,
black]{$e$} (v); \draw[-latex,
blue, very thick] (c) edge[out=180, in=0] node[pos=0.5, above,
black]{$\vdots$} (b); \draw[-latex, blue, very thick] (d) edge[out=0,
in=180](v); \draw[-latex, blue, very thick] (f) edge[out=0,
in=180]node[pos=0.5, below, black]{$\vdots$}(e); \draw[-latex, red,
dashed,very thick] (b) edge[out=270, in=90]node[pos=0.5, right,
black]{$g$}(v); \draw[-latex, red,
dashed,very thick] (c) edge[out=270, in=90]node[pos=0.5, right,
black]{$\dots$}(a); \draw[-latex, red, dashed,very thick] (f) edge[out=90,
in=270]node[pos=0.5, left, black]{$\dots$}(d); \draw[-latex, red,
dashed,very thick] (e) edge[out=90, in=270](v); \node[inner sep=1pt] (a1) at
(2.5,0){}; \draw[-latex, blue, very thick] (a1) edge[out=180, in=0](a);
\node[inner sep=1pt] (b1) at (0,2.5){}; \draw[-latex, red, dashed, very
thick] (b1) edge[out=270, in=90](b); \node[inner sep=1pt] (c1) at
(2.5,1.5){}; \draw[-latex, blue, very thick] (c1) edge[out=180, in=0](c);
\node[inner sep=1pt] (c2) at (1.5,2.5){}; \draw[-latex, red, dashed, very
thick] (c2) edge[out=270, in=90](c); \node[inner sep=1pt] (d1) at
(-2.5,0){}; \draw[-latex, blue, very thick] (d1) edge[out=0, in=180](d);
\node[inner sep=1pt] (e1) at (0,-2.5){}; \draw[-latex, red,dashed, very
thick] (e1) edge[out=90, in=270](e); \node[inner sep=1pt] (f1) at
(-2.5,-1.5){}; \draw[-latex, blue, very thick] (f1) edge[out=0, in=180](f);
\node[inner sep=1pt] (f2) at (-1.5,-2.5){}; \draw[-latex, red,dashed, very
thick] (f2) edge[out=90, in=270](f); \end{tikzpicture}
\end{equation*}%
Then
\begin{equation*}
(v\Lambda ^{e_{1}}\backslash \{e\})\cup \{(ef)(0,d(f)):f\in x\Lambda
^{e_{2}}\}=(v\Lambda ^{e_{1}}\backslash \{e\})\cup \{g\}
\end{equation*} contains neither $v\Lambda ^{e_{1}}$ nor $v\Lambda ^{e_{2}}$, so (E3)
fails.
\end{remark}

Now we study properties of efficient sets. 

\begin{definition}Let $\mathcal{E}\subseteq \operatorname{FE}(\Lambda ^{1})$.  Then %
\begin{equation*}
\min (\mathcal{E}):=\{E\in \mathcal{E}:F\subseteq E\text{ and }F\in \mathcal{%
E}\text{ imply }F=E\},
\end{equation*}%
and the \emph{edge satiation }of $\mathcal{E}$ is
\begin{equation*}
\widehat{\mathcal{E}}:=\{E\in \operatorname{FE}(\Lambda ^{1}):\text{there exists }%
F\in \mathcal{E}\text{ with }F\subseteq E\}.
\end{equation*}%
\end{definition}

\begin{remark}
Using the edge satiation, we provide an alternate characterisation of efficient.
A subset $\mathcal{E}$ of $\operatorname{FE}( \Lambda ^{1}) $ is \emph{efficient} if
it satisfies (E1) and

\begin{enumerate}
\item[(E2$^{\prime }$)] if $E\in \mathcal{E}$ and $f\in r(E)\Lambda
^{1}\backslash E$, then $\operatorname{Ext}_{\Lambda }(f;E)\in \widehat{\mathcal{E}}$%
;

\item[(E3$^{\prime }$)] if $E\in \mathcal{E}$, $e\in E$, and $F\in s(e)%
\mathcal{E}$, then $(E\backslash \{e\})\cup \{(ef)(0,d(f)):f\in F\}\in
\widehat{\mathcal{E}}$.
\end{enumerate}
\end{remark}

\begin{lemma}
\label{efficient-min-e-hat}Suppose that $\mathcal{E}\subseteq \operatorname{FE}%
(\Lambda ^{1})$ is efficient. Then $\mathcal{E}=\min (\widehat{\mathcal{E}})$%
.
\end{lemma}

\begin{proof}
To show ${\mathcal{E}}\subseteq \min (\widehat{{\mathcal{E}}})$, take $E\in {%
\mathcal{E}}$. Then $E\in \widehat{{\mathcal{E}}}$. To show $E\in \min (%
\widehat{{\mathcal{E}}})$, take $F\in \widehat{{\mathcal{E}}}$ such that $%
F\subseteq E$. Since $F\in \widehat{{\mathcal{E}}}$, there exists $F^{\prime
}\in {\mathcal{E}}$ such that $F^{\prime }\subseteq F$. So $F^{\prime
}\subseteq F\subseteq E$ and by (E1), $F^{\prime }=E$ and $E=F$. Therefore $%
E\in $ $\min (\widehat{{\mathcal{E}}})$ and ${\mathcal{E}}\subseteq \min (%
\widehat{{\mathcal{E}}})$.

To show $\min (\widehat{{\mathcal{E}}})\subseteq {\mathcal{E}}$, take $E\in
\min (\widehat{{\mathcal{E}}})$. Then $E\in \widehat{{\mathcal{E}}}$ and
there exists $F\in {\mathcal{E}}$ with $F\subseteq E$. So $F\in \widehat{{%
\mathcal{E}}}$. Since $F\subseteq E$ and $E\in \min (\widehat{{\mathcal{E}}}%
) $, $F=E$ and $E\in {\mathcal{E}}$. So $\min (\widehat{{\mathcal{E}}}%
)\subseteq {\mathcal{E}}$.
\end{proof}

The next Proposition shows that the edge satiation gives the same relative Cuntz-Krieger algebra.
\begin{proposition}
\label{cRK-family-without-efficient}Let $\Lambda $ be a
finitely aligned $k$-graph.  Suppose that $\mathcal{E}\subseteq \operatorname{%
FE}(\Lambda ^{1})$. Then $C^{\ast }(\Lambda ;\mathcal{E})=C^{\ast }(\Lambda ;%
\widehat{\mathcal{E}})$.
\end{proposition}

\begin{proof}
Since $\mathcal{E}\subseteq \widehat{\mathcal{E}}$, a relative Cuntz-Krieger
$(\Lambda ;\widehat{\mathcal{E}})$-family is a relative Cuntz-Krieger $%
(\Lambda ;\mathcal{E})$-family. Now suppose that $\{S_{\lambda }:\lambda \in
\Lambda \}$ is a relative Cuntz-Krieger $(\Lambda ;\mathcal{E})$-family. For
$E\in \widehat{\mathcal{E}}$, there exists $F\in \mathcal{E}$ with $%
F\subseteq E$, so $\prod_{e\in F}(S_{r(E)}-S_{e}S_{e}^{\ast })=0$ and
\begin{equation*}
\prod_{e\in E}(S_{r(E)}-S_{e}S_{e}^{\ast })=\prod_{e\in
F}(S_{r(E)}-S_{e}S_{e}^{\ast })\prod_{e\in \left. E\right\backslash
F}(S_{r(E)}-S_{e}S_{e}^{\ast })=0\text{.}
\end{equation*}%
Thus $\{S_{\lambda }:\lambda \in \Lambda \}$ is also a relative
Cuntz-Krieger ($\Lambda ;\widehat{\mathcal{E}})$-family. The universal
property of $C^{\ast }(\Lambda ;\mathcal{E})$ and $C^{\ast }(\Lambda ;%
\widehat{\mathcal{E}})$ implies the two algebras coincide.
\end{proof}

\section{$\mathcal{E}$-boundary paths}

\label{Section-boundary-paths}

In this section, we discuss $\mathcal{E}$-boundary paths and investigate
properties of relative Cuntz-Krieger algebras (Proposition \ref%
{properties-of-relative-CK-algebra}).
For $k\in
\mathbb{N}$ and $m\in ( \mathbb{N\cup }\{ \infty \} ) ^{k}$, $\Omega _{k,m}$
is the $k$-graph which has vertices $\{ n\in \mathbb{N}%
^{k}:n\leq m\} $, morphisms $\{ ( n_{1},n_{2}) :n_{1},n_{2}\in \mathbb{N}%
^{k},n_{1}\leq n_{2}\leq m\} $, degree map $d( ( n_{1},n_{2}) ) =n_{2}-n_{1}$
and range and source maps $r( ( n_{1},n_{2}) ) =n_{1},s( ( n_{1},n_{2}) )
=n_{2}$ (see \cite[Section 2]{RSY03}).

\begin{definition}
\label{compatible-boundary-paths}Suppose that $\Lambda $ is a finitely
aligned $k$-graph and that ${\mathcal{E}}\subseteq \operatorname{FE}( \Lambda ^{1}) $
is efficient. A path $x:\Omega _{k,m}\rightarrow \Lambda $ is an ${\mathcal{E}%
}$\emph{-boundary path} of $\Lambda $ if for $n\in \mathbb{N}^{k}$ such that
$n\leq m$, and $E\in x( n) {\mathcal{E}}$, there exists $e\in E$ such that $%
x( n,n+d( e) ) =e$. We denote the collection of all ${\mathcal{E}}$-boundary
paths of $\Lambda $ by $\partial ( \Lambda ;{\mathcal{E}}) $. We write $d(
x) $ for $m$ and $r( x) $ for $x( 0) $.
\end{definition}

The next two lemmas use similar arguments to
Lemma 4.4 and Lemma 4.7 of \cite{Si06} (so we omit the proofs).

\begin{lemma}
\label{compatible-boundary-path-translation}Let $\Lambda $ be a
finitely aligned $k$-graph and ${\mathcal{E}}\subseteq \operatorname{FE}(
\Lambda ^{1}) $ be efficient. Suppose that $x\in \partial ( \Lambda ;{%
\mathcal{E}}) .$
\end{lemma}

\begin{enumerate}
\item[(a)] If $n\in \mathbb{N}^{k}$ with $n\leq d( x) $, then $x( n,d( x) )
\in \partial ( \Lambda ;{\mathcal{E}}) $.

\item[(b)] If $\lambda \in \Lambda r( x) $, then $\lambda x\in \partial (
\Lambda ;{\mathcal{E}}) $.
\end{enumerate}

\begin{lemma}
\label{compatible-boundary-path-non-empty}Let $\Lambda $ be a
finitely aligned $k$-graph, ${\mathcal{E}}$ be efficient and $v\in
\Lambda ^{0}$. Then $v\partial ( \Lambda ;\mathcal{E}) \neq \emptyset $ and
for $E\in v\operatorname{FE}( \Lambda ^{1}) \backslash \widehat{{\mathcal{E}}}$, $%
\left. v\partial ( \Lambda ;{\mathcal{E}}) \right\backslash E\partial (
\Lambda ;{\mathcal{E}}) \neq \emptyset $.
\end{lemma}

Now we give a concrete example of a relative Cuntz-Krieger $( \Lambda ;\mathcal{E}) $%
-family. We use this family to prove Proposition \ref%
{properties-of-relative-CK-algebra}, which establishes properties of the
universal relative Cuntz-Krieger $( \Lambda ;\mathcal{E}) $-family.

\begin{example}
For a finitely aligned $k$-graph $\Lambda $ and an efficient set ${\mathcal{%
E\subseteq }}\operatorname{FE}(\Lambda ^{1})$, we define a partial isometries $%
\{S_{\lambda }^{{\mathcal{E}}}:\lambda \in \Lambda \}\subseteq \mathcal{B}%
(l^{2}(\partial (\Lambda ;{\mathcal{E}})))$ by%
\begin{equation*}
S_{\lambda }^{{\mathcal{E}}}(e_{x}):=%
\begin{cases}
e_{\lambda x} & \text{if }s(\lambda )=r(x)\text{,} \\
0 & \text{otherwise.}%
\end{cases}%
\end{equation*}%
Then an argument similar to the proof of \cite[Lemma 4.6]{Si06} 
shows $\{S_{\lambda }^{{\mathcal{E}}}:\lambda \in \Lambda \}$ is a relative
Cuntz-Krieger $(\Lambda ;{\mathcal{E}})$-family. We call this family the ${\mathcal{E}}$\emph{%
-boundary path representation of} $C^{\ast }(\Lambda ;\mathcal{E})$.
\end{example}

\begin{proposition}
\label{properties-of-relative-CK-algebra}Let $\Lambda $ be a
finitely aligned $k$-graph and $\mathcal{E}\subseteq \operatorname{FE}(\Lambda
^{1})$ be efficient. Suppose that $\{s_{\lambda }^{\mathcal{E}}:\lambda \in
\Lambda \}$ is the universal relative Cuntz-Krieger $(\Lambda ;\mathcal{E})$%
-family.

\begin{enumerate}
\item\label{it1:propsofrel} For $v\in \Lambda ^{0}$, $s_{v}^{\mathcal{E}}\neq 0$.

\item\label{it2:propsofrel} For $v\in \Lambda ^{0}$ and $E\subseteq v\Lambda ^{1}$, $E\in
\widehat{\mathcal{E}}$ if and only if $\prod_{e\in E}(s_{v}^{\mathcal{E}%
}-s_{e}^{\mathcal{E}}s_{e}^{\mathcal{E}\ast })=0$.
\end{enumerate}
\end{proposition}

\begin{proof}
We use $\{S_{\lambda }^{{\mathcal{E}}}:\lambda \in
\Lambda \}$ the ${\mathcal{E}}$-boundary path representation of $C^{\ast
}(\Lambda ;\mathcal{E})$. For \eqref{it1:propsofrel} take $v\in \Lambda ^{0}$ and $x\in
v\partial ( \Lambda ;\mathcal{E}) $. Then $S_{v}^{{\mathcal{E}}}( e_{x})
=e_{x}\neq 0$. Therefore $S_{v}^{{\mathcal{E}}}$ is nonzero and the
universal property of $C^{\ast }(\Lambda ;\mathcal{E})$ shows that $s_{v}^{{%
\mathcal{E}}}$ is nonzero.

For \eqref{it2:propsofrel}, take $v\in \Lambda ^{0}$ and $E\subseteq v\Lambda ^{1}$. Suppose $%
E\notin \widehat{\mathcal{E}}$. We give separate arguments for $E\notin
\operatorname{FE}( \Lambda ^{1}) $ and $E\in \operatorname{FE}( \Lambda ^{1}) \backslash
\widehat{\mathcal{E}}$. First suppose $E\notin \operatorname{FE}( \Lambda ^{1}) $. Then
there exists $g\in v\Lambda ^{1}$ with $\operatorname{Ext}_{\Lambda }(g;E)=\emptyset
$. For $e\in E$, we have $\Lambda ^{\min }(g,e)=\emptyset $ and $s_{g}^{%
\mathcal{E}\ast }s_{e}^{\mathcal{E}}=0$. On the other hand, by part \eqref{it1:propsofrel}, $%
s_{g}^{\mathcal{E}\ast }s_{g}^{\mathcal{E}}=s_{s( g) }^{\mathcal{E}}\neq 0$
and $s_{g}^{\mathcal{E}}\neq 0$. Since$\ s_{g}^{\mathcal{E}}s_{g}^{\mathcal{E%
}\ast }s_{g}^{\mathcal{E}}=s_{g}^{\mathcal{E}}$ is nonzero, $s_{g}^{\mathcal{%
E}}s_{g}^{\mathcal{E}\ast }\neq 0$. Thus%
\begin{equation*}
s_{g}^{\mathcal{E}}s_{g}^{\mathcal{E}\ast }\prod_{e\in E}(s_{v}^{\mathcal{E}%
}-s_{e}^{\mathcal{E}}s_{e}^{\mathcal{E}\ast })=s_{g}^{\mathcal{E}}s_{g}^{%
\mathcal{E}\ast }\neq 0
\end{equation*}%
and hence $\prod_{e\in E}(s_{v}^{\mathcal{E}}-s_{e}^{\mathcal{E}}s_{e}^{\mathcal{E}%
\ast })$ is nonzero. 

Next suppose $E\in \operatorname{FE}( \Lambda ^{1}) \backslash
\widehat{\mathcal{E}}$. By Lemma \ref{compatible-boundary-path-non-empty},
there exists $x\in \left. v\partial ( \Lambda ;\mathcal{E}) \right\backslash
E\partial ( \Lambda ;\mathcal{E}) $. For $e\in E$, $(S_{v}^{\mathcal{E}%
}-S_{e}^{\mathcal{E}}S_{e}^{\mathcal{E}\ast })( e_{x}) =S_{v}^{{\mathcal{E}}%
}( e_{x}) =e_{x}$ and
\begin{equation*}
\prod_{e\in E}(S_{v}^{\mathcal{E}}-S_{e}^{\mathcal{E}}S_{e}^{\mathcal{E}\ast
})( e_{x}) =e_{x}\neq 0.
\end{equation*}%
So $\prod_{e\in E}(S_{v}^{\mathcal{E}}-S_{e}^{\mathcal{E}}S_{e}^{\mathcal{E}%
\ast })\neq 0$ and then $\prod_{e\in E}(s_{v}^{\mathcal{E}}-s_{e}^{\mathcal{E%
}}s_{e}^{\mathcal{E}\ast })\neq 0$.

For the reverse implication suppose $E\in \widehat{\mathcal{E}}$. 
Then there exists $F\in \mathcal{E}$ with $F\subseteq E$. So $\prod_{f\in
F}(s_{v}^{\mathcal{E}}-s_{f}^{\mathcal{E}}s_{f}^{\mathcal{E}\ast })=0$ and
since $F\subseteq E$, $\prod_{e\in E}(s_{v}^{\mathcal{E}}-s_{e}^{\mathcal{E}%
}s_{e}^{\mathcal{E}\ast })=0$.
\end{proof}

\section{Relationship between efficient and satiated sets}

\label{Section-efficient-and-satiated}

Now we discuss satiated set and show that for a
finitely aligned $k$-graph, there exists a bijection between its efficient
sets and its satiated sets (Theorem \ref{bijection-efficient-and-satiated}).
As in \cite[Definition 4.1]{Si06}, a subset $\mathcal{F}\subseteq \operatorname{FE}(
\Lambda ) $ is \emph{satiated} if it satisfies
\begin{enumerate}
\item[(S1)] if $E\in \mathcal{F} $ and $F\in \operatorname{FE}( \Lambda ) $ with $%
E\subseteq F$, then $F\in \mathcal{F} $;

\item[(S2)] if $E\in \mathcal{F} $ and $\lambda \in r( E) \Lambda \backslash
E\Lambda $, then $\operatorname{Ext}_{\Lambda }( \lambda ;E) \in \mathcal{F} $;

\item[(S3)] if $E\in \mathcal{F} $ and $0<n_{\lambda }\leq d( \lambda ) $
for $\lambda \in E$, then $\{ \lambda ( 0,n_{\lambda }) :\lambda \in E\} \in
\mathcal{F} $;

\item[(S4)] if $E\in \mathcal{F}$, $E^{\prime }\subseteq E$ and for each $%
\lambda \in E^{\prime }$, $E_{\lambda }^{\prime }\in s( \lambda ) \mathcal{F}
$, then%
\begin{equation*}
\Big(( E\backslash E^{\prime }) \cup \Big(\bigcup_{\lambda \in E^{\prime
}}\lambda E_{\lambda }^{\prime }\Big)\Big)\in \mathcal{F}\text{.}
\end{equation*}
\end{enumerate}
For $\mathcal{F}\subseteq \operatorname{FE}( \Lambda ) $, Sims writes $\overline{%
\mathcal{F}}$ for the smallest satiated subset of $\operatorname{FE}( \Lambda ) $
which contains $\mathcal{F}$, and call it the \emph{satiation}\ of $\mathcal{%
F}$ \cite[Section 5]{Si06}. He also shows how to construct the satiation of $%
\mathcal{F}$. 

\begin{remark}
\label{remark-ehat-esatiated}Suppose $\mathcal{F}\subseteq \operatorname{FE}(
\Lambda ) $. Sims defines maps $\Sigma _{1}$ to $\Sigma _{4}$ \cite[Definition
5.2]{Si06} and shows that the iterated application of these maps produces 
$\overline{\mathcal{F}}$ \cite[Proposition 5.5]{Si06}. The set%
\begin{equation*}
\Sigma _{1}( \mathcal{F}) :=\{ F\in \operatorname{FE}( \Lambda ) :\text{there exists
}G\in \mathcal{F}\text{ with }G\subseteq F\}
\end{equation*}%
is contained in $\overline{\mathcal{F}}$. So we get the following result:
Suppose that ${\mathcal{E}}\subseteq \operatorname{FE}( \Lambda ^{1}) $. Then the set $\widehat{\mathcal{E}}$ is contained in $\Sigma _{1}(
\mathcal{E}) $ and $\widehat{\mathcal{E}}\subseteq \overline{\mathcal{E}}$.
So $\mathcal{E}\subseteq \widehat{\mathcal{E}}\subseteq \overline{\mathcal{%
E}}$ and $\overline{\mathcal{E}}\subseteq \overline{\widehat{\mathcal{E}}}%
\subseteq \overline{\overline{\mathcal{E}}}=\overline{\mathcal{E}}$ giving
\begin{equation*}
\overline{\widehat{\mathcal{E}}}=\overline{\mathcal{E}}.
\end{equation*}
\end{remark}

\begin{remark}
Corollary 5.6 of \cite{Si06} shows that $C^{\ast }(\Lambda ;\mathcal{E}%
)=C^{\ast }(\Lambda ;\overline{\mathcal{E}})$. So by Proposition \ref%
{cRK-family-without-efficient}, $C^{\ast }(\Lambda ;\mathcal{E})$, $%
C^{\ast }(\Lambda ;\widehat{\mathcal{E}})$ and $C^{\ast }(\Lambda ;\overline{%
\mathcal{E}})$ all coincide.
\end{remark}

Now we state the main result of this section.
\begin{theorem}
\label{bijection-efficient-and-satiated}For a finitely aligned $k$-graph $%
\Lambda $, the map ${\mathcal{E}}\mapsto \overline{{\mathcal{E}}}$ is a
bijection between efficient sets of $\Lambda $ and satiated sets of $\Lambda
$, with inverse given by $\mathcal{F}\mapsto \min ( \mathcal{F}\cap \operatorname{FE}%
( \Lambda ^{1}) ) $.
\end{theorem}

The rest of this section is devoted to proving Theorem \ref%
{bijection-efficient-and-satiated}. First we establish some preliminary
results (Proposition \ref{connect-efficient-to-satiated} and Proposition \ref%
{connect-satiated-to-efficient}). Proposition \ref%
{connect-efficient-to-satiated} describes the relationship between an efficient
set with its satiation. On the other hand, given a satiated set, we
construct an efficient set in Proposition \ref{connect-satiated-to-efficient}%
.

\begin{proposition}
\label{connect-efficient-to-satiated}Suppose that $\Lambda $ is a finitely
aligned $k$-graph and that ${\mathcal{E}}\subseteq \operatorname{FE}( \Lambda ^{1}) $
is efficient. Then $\widehat{{\mathcal{E}}}=\overline{{\mathcal{E}}}\cap
\operatorname{FE}( \Lambda ^{1}) $.
\end{proposition}

\begin{proof}
To show $\widehat{{\mathcal{E}}}\subseteq ( \overline{{\mathcal{E}}}\cap
\operatorname{FE}( \Lambda ^{1}) ) $, take $E\in \widehat{{\mathcal{E}}}$. It is
clear that $E\in \operatorname{FE}( \Lambda ^{1}) $. On the other hand, by Remark %
\ref{remark-ehat-esatiated}, $\widehat{{\mathcal{E}}}\subseteq \overline{{%
\mathcal{E}}}$ and $E\in \overline{{\mathcal{E}}}$. So $\widehat{{\mathcal{E}%
}}\subseteq ( \overline{{\mathcal{E}}}\cap \operatorname{FE}( \Lambda ^{1}) ) $.

For $(\overline{{\mathcal{E}}}\cap \operatorname{FE}(\Lambda ^{1}))\subseteq
\widehat{{\mathcal{E}}}$, take $E\in (\overline{{\mathcal{E}}}\cap \operatorname{FE}%
(\Lambda ^{1}))$. Let $\{s_{\lambda }^{{\mathcal{E}}}:\lambda \in \Lambda \}$
be the universal relative Cuntz-Krieger $(\Lambda ;{\mathcal{E}})$-family.
So $\prod_{e\in E}(s_{r(E)}^{{\mathcal{E}}}-s_{e}^{{\mathcal{E}}}s_{e}^{{%
\mathcal{E}}\ast })=0$ since $E\in \overline{{\mathcal{E}}}$ and \cite[%
Corollary 4.9]{Si06}. Since $E\in \operatorname{FE}(\Lambda ^{1})$, Proposition %
\ref{properties-of-relative-CK-algebra}\eqref{it2:propsofrel} implies $E\in \widehat{{%
\mathcal{E}}}$.
\end{proof}

Before stating Proposition \ref{connect-satiated-to-efficient}, we establish
some results that we use in the proof.

\begin{lemma}
\label{efficient-inequality-lemma}Let $\Lambda $ be a finitely
aligned $k$-graph. Suppose that $\lambda \in \Lambda $, that $m\in \mathbb{N}%
^{k}$, and that $E\subseteq r(\lambda )\Lambda ^{m}$. If $\operatorname{Ext}%
_{\Lambda }(\lambda ;E)\neq \emptyset $, then there exists a unique $n\in
\mathbb{N}^{k}$ such that $\operatorname{Ext}_{\Lambda }(\lambda ;E)\subseteq
s(\lambda )\Lambda ^{n}$ and $\left\vert n\right\vert \leq \left\vert
m\right\vert $.
\end{lemma}

\begin{proof}
Take $\nu _{1},\nu _{2}\in \operatorname{Ext}_{\Lambda }(\lambda ;E)$. We show $%
d(\nu _{1})=d(\nu _{2})$. So there exist $\mu _{1},\mu _{2}\in E$ with $\nu
_{1}\in \operatorname{Ext}_{\Lambda }(\lambda ;\{\mu _{1}\})$ and $\nu _{2}\in \operatorname{%
Ext}_{\Lambda }(\lambda ;\{\mu _{2}\})$. Since $E\subseteq r(\lambda
)\Lambda ^{m}$, then $d(\mu _{1})=d(\mu _{2})$ and%
\begin{equation*}
d(\lambda )+d(\nu _{1})=d(\lambda \nu _{1})=d(\lambda )\vee d(\mu
_{1})=d(\lambda )\vee d(\mu _{2})=d(\lambda \nu _{2})=d(\lambda )+d(\nu _{2})%
\text{.}
\end{equation*}%
So $d(\nu _{1})=d(\nu _{2})$. Then there exists a unique $n\in \mathbb{N}%
^{k} $ such that $\operatorname{Ext}_{\Lambda }(\lambda ;E)\subseteq s(\lambda
)\Lambda ^{n}$.

To show $\left\vert n\right\vert \leq \left\vert m\right\vert $, take $\nu
\in \operatorname{Ext}_{\Lambda }(\lambda ;E)$. There exists $\mu \in E\subseteq
r(\lambda )\Lambda ^{m}$ with $\nu \in \operatorname{Ext}_{\Lambda }(\lambda ;\{\mu
\})$. Then $d(\mu )=m$ and%
\begin{equation*}
\left\vert d(\lambda )\right\vert +\left\vert n\right\vert =\left\vert
d(\lambda )+d(\nu )\right\vert =\left\vert d(\lambda )\vee d(\mu
)\right\vert \leq \left\vert d(\lambda )\right\vert +\left\vert d(\mu
)\right\vert =\left\vert d(\lambda )\right\vert +\left\vert m\right\vert
\text{.}
\end{equation*}%
Thus $\left\vert n\right\vert \leq \left\vert m\right\vert $.
\end{proof}

\begin{corollary}
\label{efficient-inequality}For $E\subseteq \Lambda $, we define%
\begin{equation*}
L(E):=\max_{\mu \in E}\left\vert d(\mu )\right\vert
\end{equation*}%
($L(E):=0$ if $E=\emptyset $). Then for $v\in \Lambda ^{0}$, $\lambda \in
v\Lambda $ and $E\subseteq v\Lambda $, $L(\operatorname{Ext}_{\Lambda }(\lambda
;E))\leq L(E)$.
\end{corollary}

\begin{proof}
If $\operatorname{Ext}_{\Lambda }(\lambda ;E)=\emptyset $, then we are done. Suppose
$\operatorname{Ext}_{\Lambda }(\lambda ;E)\neq \emptyset $. Take $\mu \in \operatorname{Ext}%
_{\Lambda }(\lambda ;E)$ with $\left\vert d(\mu )\right\vert =L(\operatorname{Ext}%
_{\Lambda }(\lambda ;E))$. So there exists $\gamma \in E$ such that $\mu \in
\operatorname{Ext}_{\Lambda }(\lambda ;\{\gamma \})$. By Lemma \ref%
{efficient-inequality-lemma}, $\left\vert d(\mu )\right\vert \leq \left\vert
d(\gamma )\right\vert \leq L(E)$. Therefore $L(\operatorname{Ext}_{\Lambda }(\lambda
;E))\leq L(E).$
\end{proof}

\begin{lemma}
\label{connect-satiated-to-efficient-lemma}Suppose that $\Lambda $ is a
finitely aligned $k$-graph\ and that $\mathcal{F}$ is a satiated set. Then $%
\mathcal{F}\subseteq \overline{(\mathcal{F}\cap \operatorname{FE}(\Lambda ^{1}))}$.
\end{lemma}

\begin{proof}
For $E\subseteq \Lambda $ and $l\in \mathbb{N}$, we define%
\begin{equation*}
N(E;l):=\left\vert \{m\in \mathbb{N}^{k}:\left\vert m\right\vert =l\text{
and there exists }\mu \in E\text{ with }d(\mu )=m\}\right\vert \text{.}
\end{equation*}%
With a slight abuse of notation, $N(E):=N(E;L(E))$.  To prove the lemma, we show that for 
$E\in\operatorname{FE}(\Lambda )$,
\begin{equation}
E\in \mathcal{F}\ \text{implies }E\in \overline{(\mathcal{F}\cap \operatorname{FE}%
(\Lambda ^{1}))}\text{.}  \label{eq1}
\end{equation}%
We use nested induction arguments on pairs in $(L(E),N(E))$.  Our strategy is as
follows:
We start out by proving that \eqref{eq1}\
is true for $(L(E),N(E))=(l,j)$ for $j\in \mathbb{N}$ by induction on $l$. 
So step 1 is to show that \eqref{eq1}\ is true for $(L(E),N(E))=(1,j)$.
Then for the inductive step, we assume that $l \geq 2$ and  \eqref{eq1} is true for $(L(E),N(E))=(l-x,j)$ for all $j$
and $1\leq x\leq l-1$.  We prove that \eqref{eq1} is true for $(L(E),N(E))=(l,j)$ by induction on $l$.
Thus in step 2 we show \eqref{eq1} is true for $(L(E),N(E))=(l,1)$ (using the inductive hypothesis for $l$).
Then we assume \eqref{eq1} is true for $(L(E),N(E))=(l,j-y)$ with $1\leq y\leq j-1$.
Finally, step 3 is to show that, with these assumptions in place, \eqref{eq1} holds for $(L(E),N(E))=(l,j)$.

Step 1:  If $E\in \mathcal{F}$ with $L(E)=1$, then $E\in
\operatorname{FE}(\Lambda ^{1})$ and $E\in (\mathcal{F}\cap \operatorname{FE}(\Lambda ^{1}))$%
. So $E\in \overline{(\mathcal{F}\cap \operatorname{FE}(\Lambda ^{1}))}$, as
required.

Since the argument for Step 2 and Step 3 is similar, to save from repeating things we  
take $E\in \mathcal{F}$ with $(L(E),N(E))=(l,j)$ where either $j=1$ (Step 2) or $j\geq 2$ (Step 3). To show 
$E\in \overline{(\mathcal{F}\cap \operatorname{FE}(\Lambda ^{1}))}$, we prove that $E$ can be
constructed by applying processes in (S1-4) to certain elements of 
$\overline{(\mathcal{F}\cap \operatorname{FE}(\Lambda ^{1}))}$.

Since $(L(E),N(E))=(l,j)$, there exist $m_{1},\ldots ,m_{j}\in \mathbb{N}%
^{k} $ such that for $1\leq i\leq j$, we have $\left\vert m_{i}\right\vert =l
$ and there exists $\lambda _{i}\in E$ with $d(\lambda _{i})=m_{i}$. Define
\begin{equation*}
E_{m_{i}}:=\{\lambda \in E:d(\lambda )=m_{i}\}\text{ for }1\leq i\leq j\text{%
,}
\end{equation*}%
\begin{equation*}
E^{\prime }:=\{\lambda \in E_{m_{1}}:\text{there exists }\nu _{\lambda }\in E%
\text{ with }\nu _{\lambda }\neq \lambda \text{ and }\nu _{\lambda }\nu
_{\lambda }^{\prime }=\lambda \}\text{,}
\end{equation*}%
\begin{equation*}
E^{\prime \prime }:=E_{m_{1}}\backslash E^{\prime }\text{.}
\end{equation*}%
For $\lambda \in E^{\prime \prime }$, we choose $i_{\lambda }$ with $%
d(\lambda )\geq e_{i_{\lambda }}$ and define $\lambda ^{\prime }:=\lambda
(0,d(\lambda )-e_{i_{\lambda }})$.  Since $l \geq 2$, then $\lambda ^{\prime } \notin \Lambda^{0}$. Now we establish the following claims:

\begin{claim}
\label{claim-connect-1}$\left. E\right\backslash E^{\prime }\in \mathcal{F}$.
\end{claim}

\begin{proof}[Proof of Claim~\protect\ref{claim-connect-1}]
For $\lambda \in E^{\prime }$, we get $\lambda \in E_{m_{1}}\subseteq
r(\lambda )\Lambda ^{m_{1}}$, $\nu _{\lambda }\notin E^{\prime }$ and $\nu
_{\lambda }\in \left. E\right\backslash E^{\prime }$. So%
\begin{equation*}
E\backslash E^{\prime }=E\backslash E^{\prime }\cup \{\nu _{\lambda
}:\lambda \in E^{\prime }\}=\{\lambda :\lambda \in E\backslash E^{\prime
}\}\cup \{\lambda (0,d(\nu _{\lambda })):\lambda \in E^{\prime }\}\text{.}
\end{equation*}%
Since $E\in \mathcal{F}$ and $\mathcal{F}$ is satiated, then by (S3), $%
\left. E\right\backslash E^{\prime }\in \mathcal{F}$.\hfil\penalty100\hbox{}%
\nobreak\hfill \hbox{\qed\
Claim~\ref{claim-connect-1}} \renewcommand\qed{}
\end{proof}

\begin{claim}
\label{claim-connect-2}For $\lambda \in E^{\prime \prime }$, we have $%
\lambda ^{\prime }\notin E\Lambda $ and $\operatorname{Ext}_{\Lambda }(\lambda
^{\prime };\left. E\right\backslash E^{\prime })\in \overline{(\mathcal{F}%
\cap \operatorname{FE}(\Lambda ^{1}))}$.
\end{claim}

\begin{proof}[Proof of Claim~\protect\ref{claim-connect-2}]
Take $\lambda \in E^{\prime \prime }$. Since $\lambda \in E^{\prime \prime
}=E_{m_{1}}\backslash E^{\prime }$, then $\lambda \notin E\Lambda \backslash
E$. Suppose for contradiction that $\lambda ^{\prime }\in E\Lambda $. Write $%
\lambda ^{\prime }:=e\mu $ with $e\in E$ and $\mu \in \Lambda $. Then
\begin{equation*}
\lambda =\lambda ^{\prime }\left[ \lambda (d(\lambda )-e_{i_{\lambda
}},d(\lambda ))\right] =e\mu \left[ \lambda (d(\lambda )-e_{i_{\lambda
}},d(\lambda ))\right]
\end{equation*}%
and $\lambda \in E\Lambda \backslash E$, which contradicts $\lambda \notin
E\Lambda \backslash E$. Thus $\lambda ^{\prime }\notin E\Lambda $.

Since $\lambda ^{\prime }\notin E\Lambda $, we have $\lambda ^{\prime
}\notin (E\backslash E^{\prime })\Lambda $ and $\lambda ^{\prime }\in
r(E\backslash E^{\prime })\Lambda \backslash (E\backslash E^{\prime
})\Lambda $. Since $E\backslash E^{\prime }\in \mathcal{F}$ (Claim \ref%
{claim-connect-1}), by (S2), $\operatorname{Ext}_{\Lambda }(\lambda ^{\prime
};E\backslash E^{\prime })\in \mathcal{F}$. To show $\operatorname{Ext}_{\Lambda
}(\lambda ^{\prime };E\backslash E^{\prime })\in \overline{(\mathcal{F}\cap
\operatorname{FE}(\Lambda ^{1}))}$, we now give separate arguments for $j=1$ (step 2) and $j\geq
1$ (step 3).

\begin{enumerate}
\item[(Step 2)] Suppose $j=1$. Then for $\nu \in E\backslash E_{m_{1}}$, we
have $\left\vert d(\nu )\right\vert \leq l-1$ and by Corollary \ref%
{efficient-inequality}, $L(\operatorname{Ext}_{\Lambda }(\lambda ^{\prime };\{\nu
\}))\leq l-1$. Since $(E\backslash E^{\prime })\backslash
E_{m_{1}}=E\backslash E_{m_{1}}$, this implies%
\begin{equation*}
L(\operatorname{Ext}_{\Lambda }(\lambda ^{\prime };(E\backslash E^{\prime
})\backslash E_{m}))\leq l-1\text{.}
\end{equation*}%
For $\mu \in E_{m_{1}}\subseteq r(E)\Lambda ^{m_{1}}$, we have $d(\lambda
^{\prime })=m_{1}-e_{i_{\lambda }}=d(\mu )-e_{i_{\lambda }}$, so $(d(\lambda
^{\prime })\vee d(\mu ))-d(\lambda ^{\prime })=e_{i_{\lambda }}$ and $L(%
\operatorname{Ext}_{\Lambda }(\lambda ^{\prime };\{\mu \}))\leq 1$. Then $L(\operatorname{Ext%
}_{\Lambda }(\lambda ^{\prime };E_{m}))\leq 1$ and%
\begin{equation*}
L(\operatorname{Ext}_{\Lambda }(\lambda ^{\prime };E\backslash E^{\prime }))=\max
\{L(\operatorname{Ext}_{\Lambda }(\lambda ^{\prime };(E\backslash E^{\prime
})\backslash E_{m}))),L(\operatorname{Ext}_{\Lambda }(\lambda ^{\prime
};E_{m}))\}\leq l-1\text{.}
\end{equation*}%
Since $\operatorname{Ext}_{\Lambda }(\lambda ^{\prime };E\backslash E^{\prime })\in
\mathcal{F}$ and $L(\operatorname{Ext}_{\Lambda }(\lambda ^{\prime };E\backslash
E^{\prime }))\leq l-1$, by the inductive hypothesis for $l$, we have $\operatorname{Ext}%
_{\Lambda }(\lambda ^{\prime };E\backslash E^{\prime })\in \overline{(%
\mathcal{F}\cap \operatorname{FE}(\Lambda ^{1}))}$, as required.

\item[(Step 3)] Since we have now verified both bases cases,
we have both $l\geq 2$ and $j\geq 2$. 
Take $2\leq i\leq j$. For $\nu \in E_{m_{i}}$, we have $\left\vert d(\nu )\right\vert =\left\vert
m_{i}\right\vert =l$ and by Corollary \ref{efficient-inequality}, $L(\operatorname{%
Ext}_{\Lambda }(\lambda ^{\prime };\{\nu \}))\leq \left\vert d(\nu
)\right\vert =l$. So $L(\operatorname{Ext}_{\Lambda }(\lambda ^{\prime
};E_{m_{i}}))\leq l$. If $\operatorname{Ext}_{\Lambda }(\lambda ^{\prime
};E_{m_{i}})=\emptyset $, then $N(\operatorname{Ext}_{\Lambda }(\lambda ^{\prime
};E_{m_{i}});l)=0$, otherwise, since $E_{m_{i}}\subseteq s(\lambda )\Lambda
^{m_{i}}$, by Lemma \ref{efficient-inequality-lemma}, there exists a unique $%
n\in \mathbb{N}^{k}$ with $\left\vert n\right\vert \leq \left\vert
m_{i}\right\vert =l$ such that $\operatorname{Ext}_{\Lambda }(\lambda
;E_{m_{i}})\subseteq s(\lambda )\Lambda ^{n}$. Hence in either case, $N(%
\operatorname{Ext}_{\Lambda }(\lambda ^{\prime };E_{m_{i}});l)\leq 1$. Therefore%
\begin{equation}
L(\operatorname{Ext}_{\Lambda }(\lambda ^{\prime };\bigcup_{2\leq i\leq
j}E_{m_{i}}))\leq l\text{ and }N(\operatorname{Ext}_{\Lambda }(\lambda ^{\prime
};\bigcup_{2\leq i\leq j}E_{m_{i}});l)\leq \sum_{2\leq i\leq j}1=j-1\text{.}
\label{equ-efficient-1}
\end{equation}%
On the other hand, for $\mu \in E_{m_{1}}\subseteq r(E)\Lambda ^{m_{1}}$, we
have $d(\lambda ^{\prime })=m_{1}-e_{i_{\lambda }}=d(\mu )-e_{i_{\lambda }}$%
, so $(d(\lambda ^{\prime })\vee d(\mu ))-d(\lambda ^{\prime
})=e_{i_{\lambda }}$ and $L(\operatorname{Ext}_{\Lambda }(\lambda ^{\prime };\{\mu
\}))\leq 1$. Thus%
\begin{equation}
L(\operatorname{Ext}_{\Lambda }(\lambda ^{\prime };E_{m_{1}}))\leq 1\text{ and }N(%
\operatorname{Ext}_{\Lambda }(\lambda ^{\prime };E_{m_{1}});l)=0
\label{equ-efficient-2}
\end{equation}%
since $l\geq 2$. Now note that for $\nu \in (E\backslash E^{\prime
})\backslash \bigcup_{1\leq i\leq j}E_{m_{i}}$, we have $\left\vert d(\nu
)\right\vert \leq l-1$ and by Corollary \ref{efficient-inequality}, $L(\operatorname{%
Ext}_{\Lambda }(\lambda ^{\prime };\{\nu \})\leq l-1$. Hence%
\begin{equation}
L(\operatorname{Ext}_{\Lambda }(\lambda ^{\prime };(E\backslash E^{\prime
})\backslash \bigcup_{1\leq i\leq j}E_{m_{i}}))\leq l-1\text{ and }N(\operatorname{%
Ext}_{\Lambda }(\lambda ^{\prime };(E\backslash E^{\prime })\backslash
\bigcup_{1\leq i\leq j}E_{m_{i}});l)=0.  \label{equ-efficient-3}
\end{equation}%
Therefore by \eqref{equ-efficient-1}, \eqref{equ-efficient-2}, and %
\eqref{equ-efficient-3}, we get
\begin{align*}
L(\operatorname{Ext}_{\Lambda }(\lambda ^{\prime };E\backslash E^{\prime }))& =\max
\{%
\begin{array}{c}
L(\operatorname{Ext}_{\Lambda }(\lambda ^{\prime };\bigcup_{2\leq i\leq
j+1}E_{m_{i}})),L(\operatorname{Ext}_{\Lambda }(\lambda ^{\prime };E_{m_{1}})), \\
L(\operatorname{Ext}_{\Lambda }(\lambda ^{\prime };(E\backslash E^{\prime
})\backslash \bigcup_{1\leq i\leq j+1}E_{m_{i}}))%
\end{array}%
\} \\
& \leq \max \{l,1,l-1\}=l\text{,}
\end{align*}%
\begin{align*}
N(\operatorname{Ext}_{\Lambda }(\lambda ^{\prime };E\backslash E^{\prime });l)& =N(%
\operatorname{Ext}_{\Lambda }(\lambda ^{\prime };\bigcup_{2\leq i\leq
j+1}E_{m_{i}});l)+N(\operatorname{Ext}_{\Lambda }(\lambda ^{\prime };E_{m_{1}});l) \\
& \text{ \ \ }+N(\operatorname{Ext}_{\Lambda }(\lambda ^{\prime };(E\backslash
E^{\prime })\backslash \bigcup_{1\leq i\leq j+1}E_{m_{i}});l) \\
& \leq (j-1)+0+0=j-1\text{.}
\end{align*}%
Hence $L(\operatorname{Ext}_{\Lambda }(\lambda ^{\prime };E\backslash E^{\prime }))$
is either equal to $l$ with $N(\operatorname{Ext}_{\Lambda }(\lambda ^{\prime
};E\backslash E^{\prime });l)\leq j-1$; or strictly less than $l$. In either
case, by the inductive hypotheses, $\operatorname{Ext}_{\Lambda }(\lambda ^{\prime
};E\backslash E^{\prime })\in \overline{(\mathcal{F}\cap \operatorname{FE}(\Lambda
^{1}))}\,$\ since $\operatorname{Ext}_{\Lambda }(\lambda ^{\prime };E\backslash
E^{\prime })\in \mathcal{F}$.
\end{enumerate}

Therefore, for $\lambda \in E^{\prime \prime }$, we have $\operatorname{Ext}%
_{\Lambda }(\lambda ^{\prime };E\backslash E^{\prime })\in \overline{(%
\mathcal{F}\cap \operatorname{FE}(\Lambda ^{1}))}$.\hfil\penalty100\hbox{}\nobreak%
\hfill \hbox{\qed\
Claim~\ref{claim-connect-2}} \renewcommand\qed{}
\end{proof}

\begin{claim}
\label{claim-connect-3}$G:=((E\backslash E^{\prime })\backslash E^{\prime
\prime }\cup \bigcup_{\lambda \in E^{\prime \prime }}\{\lambda ^{\prime
}\})\in \overline{(\mathcal{F}\cap \operatorname{FE}(\Lambda ^{1}))}$.
\end{claim}

\begin{proof}[Proof of Claim~\protect\ref{claim-connect-3}]
Note that $\left. E\right\backslash E^{\prime }\in \mathcal{F}$ (Claim \ref%
{claim-connect-1}). Then by (S3),
\begin{equation*}
G=\{\lambda :\lambda \in (E\backslash E^{\prime })\backslash E^{\prime
\prime }\}\cup \{\lambda (0,d(\lambda )-e_{i_{\lambda }}):\lambda \in
E^{\prime \prime }\}\in \mathcal{F}\text{.}
\end{equation*}%
If $j=1$, then $E_{m_{1}}$ contains all paths in $E$ with absolute length $l$
and $L(G)\leq l-1$. If $j\geq 2$, then $E_{m_{1}}$ contains $E\cap \Lambda
^{m_{1}}$ with $\left\vert m_{1}\right\vert =l$, so $L(G)=l$ and $N(G)=j-1$.
In either case, by the inductive hypothesis, $G\in \overline{(\mathcal{F}%
\cap \operatorname{FE}(\Lambda ^{1}))}$.\hfil\penalty100\hbox{}\nobreak\hfill
\hbox{\qed\
Claim~\ref{claim-connect-3}} \renewcommand\qed{}
\end{proof}

Now we show $E\in \overline{(\mathcal{F}\cap \operatorname{FE}(\Lambda ^{1}))}$. For
every $\lambda \in E^{\prime \prime }$, we have $\lambda ^{\prime }\notin E$
(Claim \ref{claim-connect-2}) and then $(E\backslash E^{\prime })\backslash
E^{\prime \prime }=(G\backslash \bigcup_{\lambda \in E^{\prime \prime
}}\lambda ^{\prime })$. Because $G\in \overline{(\mathcal{F}\cap \operatorname{FE}%
(\Lambda ^{1}))}$ (Claim \ref{claim-connect-3}) and for $\lambda \in
E^{\prime \prime }$, $\operatorname{Ext}_{\Lambda }(\lambda ^{\prime };E\backslash
E^{\prime })\in \overline{(\mathcal{F}\cap \operatorname{FE}(\Lambda ^{1}))}$ (Claim %
\ref{claim-connect-2}), by (S4),
\begin{align}
F& :=(E\backslash E^{\prime })\backslash E^{\prime \prime }\cup
\bigcup_{\lambda \in E^{\prime \prime }}\lambda ^{\prime }\operatorname{Ext}%
_{\Lambda }(\lambda ^{\prime };E\backslash E^{\prime }))
\label{equ-claim-connect-F} \\
& =((G\backslash \bigcup_{\lambda \in E^{\prime \prime }}\lambda ^{\prime
})\cup \bigcup_{\lambda \in E^{\prime \prime }}\lambda ^{\prime }\operatorname{Ext}%
_{\Lambda }(\lambda ^{\prime };E\backslash E^{\prime }))\in \overline{(%
\mathcal{F}\cap \operatorname{FE}(\Lambda ^{1}))}\text{.}  \notag
\end{align}%
On the other hand, for $\nu \in \bigcup_{\lambda \in E_{m_{1}}^{\prime
}}\lambda ^{\prime }\operatorname{Ext}_{\Lambda }(\lambda ^{\prime };E\backslash
E^{\prime })$, there exists $n_{\nu }\in \mathbb{N}^{k}$ such that $\nu
(0,d(n_{\nu }))\in E\backslash E^{\prime }$. For $\nu \in (E\backslash
E^{\prime })\backslash E^{\prime \prime }$, set $n_{\nu }:=d(\nu )$. By
(S3), \eqref{equ-claim-connect-F} implies%
\begin{equation*}
\{\nu (0,n_{\nu }):\nu \in F\}\in \overline{(\mathcal{F}\cap \operatorname{FE}%
(\Lambda ^{1}))}.
\end{equation*}%
Note that $\{\nu (0,n_{\nu }):\nu \in F\}\subseteq E\backslash E^{\prime
}\subseteq E$ and by (S1), $E\in \overline{(\mathcal{F}\cap \operatorname{FE}%
(\Lambda ^{1}))}$. So \eqref{eq1} is true for $(L(E),N(E))=(l,j)$.
\end{proof}

\begin{proposition}
\label{connect-satiated-to-efficient}Suppose that $\Lambda $ is a finitely
aligned $k$-graph\ and that $\mathcal{F}$ is a satiated set. Then $\min (%
\mathcal{F}\cap \operatorname{FE}(\Lambda ^{1}))$ is efficient and $\overline{\min (%
\mathcal{F}\cap \operatorname{FE}(\Lambda ^{1}))}=\mathcal{F}$.
\end{proposition}

\begin{proof}
We show that $\min (\mathcal{F}\cap \operatorname{FE}(\Lambda ^{1}))$ is efficient.
To show (E1), take $E,F\in \min (\mathcal{F}\cap \operatorname{FE}(\Lambda ^{1}))$
with $E\subseteq F$. By definition of $\min (\mathcal{F}\cap \operatorname{FE}%
(\Lambda ^{1}))$, we have $E=F$ and (E1) holds. For (E2), take $E\in \min (%
\mathcal{F}\cap \operatorname{FE}(\Lambda ^{1}))$ and $g\in r(E)\Lambda
^{1}\backslash E$. Since $E\in \operatorname{FE}(\Lambda ^{1})$, by Corollary \ref%
{efficient-exhaustive}, $\operatorname{Ext}_{\Lambda }(g;E)\in \operatorname{FE}(\Lambda
^{1})$. Since $E\in \mathcal{F}$, by (S2), $\operatorname{Ext}_{\Lambda }(g;E)\in
\mathcal{F}$. Then $\operatorname{Ext}_{\Lambda }(g;E)\in (\mathcal{F}\cap \operatorname{FE}%
(\Lambda ^{1}))$ and there exists $F\in \min (\mathcal{F}\cap \operatorname{FE}%
(\Lambda ^{1}))$ with $F\subseteq \operatorname{Ext}_{\Lambda }(g;E)$. So (E2) also
holds. To show (E3), take $E,F\in \min (\mathcal{F}\cap \operatorname{FE}(\Lambda
^{1}))$ and $g\in Er(F)$. Define
\begin{equation*}
E_{F}:=(E\backslash \{g\})\cup \{(gf)(0,d(f)):f\in F\}\text{.}
\end{equation*}%
Since $E,F\in \mathcal{F}$, by (S4), $E_{F}\in \mathcal{F}$ and $E_{F}\in
\operatorname{FE}(\Lambda )$. Then $E_{F}\in \operatorname{FE}(\Lambda ^{1})$ since $%
E_{F}\subseteq r(E_{F})\Lambda ^{1}$. So $E_{F}\in (\mathcal{F}\cap \operatorname{FE}%
(\Lambda ^{1}))$ and there exists $G\in \min (\mathcal{F}\cap \operatorname{FE}%
(\Lambda ^{1}))$ with $G\subseteq E_{F}$. Therefore (E3) holds and $\min (%
\mathcal{F}\cap \operatorname{FE}(\Lambda ^{1}))$ is efficient.

Now we show $\overline{\min (\mathcal{F}\cap \operatorname{FE}(\Lambda ^{1}))}=%
\mathcal{F}$.  Since both $\min (%
\mathcal{F}\cap \operatorname{FE}(\Lambda ^{1}))$ and $\mathcal{F}\cap \operatorname{FE}(\Lambda ^{1})$ have the same edge satiation, then by Remark \ref{remark-ehat-esatiated}, $\overline{\min (%
\mathcal{F}\cap \operatorname{FE}(\Lambda ^{1}))}=\overline{(\mathcal{F}\cap \operatorname{FE%
}(\Lambda ^{1}))}$. So it suffices to show $\overline{(\mathcal{F}\cap \operatorname{%
FE}(\Lambda ^{1}))}=\mathcal{F}$. Note that $(\mathcal{F}\cap \operatorname{FE}%
(\Lambda ^{1}))\subseteq \mathcal{F}$. So $\overline{(\mathcal{F}\cap \operatorname{%
FE}(\Lambda ^{1}))}\subseteq \overline{\mathcal{F}}=\mathcal{F}$ since $%
\mathcal{F}$ is satiated. The other inclusion follows from Lemma \ref%
{connect-satiated-to-efficient-lemma}.
\end{proof}

We are finally ready to prove Theorem \ref{bijection-efficient-and-satiated}.

\begin{proof}[Proof of Theorem \protect\ref{bijection-efficient-and-satiated}%
]
To show injectivity, take efficient sets ${\mathcal{E}}_{1}$ and ${%
\mathcal{E}}_{2}$ with $\overline{{\mathcal{E}}_{1}}=\overline{{\mathcal{E}}%
_{2}}$. By Proposition \ref{connect-efficient-to-satiated}, $\widehat{{%
\mathcal{E}}_{1}}=\overline{{\mathcal{E}}_{1}}\cap \operatorname{FE}( \Lambda ^{1}) =%
\overline{{\mathcal{E}}_{2}}\cap \operatorname{FE}( \Lambda ^{1}) =\widehat{{%
\mathcal{E}}_{2}}$. By Lemma \ref{efficient-min-e-hat}, ${\mathcal{E}}%
_{1}={\mathcal{E}}_{2}$.
For surjectivity, take a satiated set $\mathcal{F}$. By Proposition %
\ref{connect-satiated-to-efficient}, $\min ( \mathcal{F}\cap \operatorname{FE}(
\Lambda ^{1}) ) $ is efficient and $\mathcal{F}=\overline{\min ( \mathcal{F}%
\cap \operatorname{FE}( \Lambda ^{1}) ) }$, as required.
\end{proof}

A direct consequence of Theorem \ref{bijection-efficient-and-satiated} is:

\begin{corollary}
\label{bijection-satiated-and-minimally-satiation}The map $\theta :\mathcal{F%
}\mapsto \mathcal{F}\cap \operatorname{FE}( \Lambda ^{1}) $ is a bijective map
between satiated sets and edge satiations of efficient sets. Furthermore $%
\theta $ preserves containment in the sense that $\mathcal{F}%
_{1}\varsubsetneq \mathcal{F}_{2}$ implies $( \mathcal{F}_{1}\cap \operatorname{FE}(
\Lambda ^{1}) ) \varsubsetneq ( \mathcal{F}_{2}\cap \operatorname{FE}( \Lambda ^{1})
) $.
\end{corollary}

\begin{proof}
The bijectivity of $\theta $ follows Lemma \ref%
{efficient-min-e-hat} and Theorem \ref{bijection-efficient-and-satiated}.
Take satiated sets $\mathcal{F}_{1},\mathcal{F}_{2}$ $\subseteq \operatorname{FE}(
\Lambda ) $ such that $\mathcal{F}_{1}\varsubsetneq \mathcal{F}_{2}$. We
trivially have $( \mathcal{F}_{1}\cap \operatorname{FE}( \Lambda ^{1}) ) \subseteq (
\mathcal{F}_{2}\cap \operatorname{FE}( \Lambda ^{1}) ) $. Suppose for contradiction
that $( \mathcal{F}_{1}\cap \operatorname{FE}( \Lambda ^{1}) ) =( \mathcal{F}%
_{2}\cap \operatorname{FE}( \Lambda ^{1}) ) $. Then
\begin{equation*}
\mathcal{F}_{1}=\theta ^{-1}( \mathcal{F}_{1}\cap \operatorname{FE}( \Lambda ^{1}) )
=\theta ^{-1}( \mathcal{F}_{2}\cap \operatorname{FE}( \Lambda ^{1}) ) =\mathcal{F}%
_{2}\text{,}
\end{equation*}%
which contradicts $\mathcal{F}_{1}\neq \mathcal{F}_{2}$. The conclusion
follows.
\end{proof}

\section{\protect\bigskip Applications}

\label{Section-applications}

\subsection{The gauge-invariant uniqueness theorem for relative
Cuntz-Krieger algebras}

\label{Subsection-gauge-invariant-uniqueness-theorem}In \cite{Si06}, Sims
introduced two uniqueness theorems for relative Cuntz-Krieger algebras,
namely the gauge-invariant uniqueness theorem \cite[Theorem 6.1]{Si06} and
the Cuntz-Krieger uniqueness theorem \cite[Theorem 6.3]{Si06}. In this
subsection, we show how our Theorem \ref{bijection-efficient-and-satiated}
simplifies the hypothesis of \cite[Theorem 6.1]{Si06} as follows:

\begin{theorem}[The gauge-invariant uniqueness theorem]
\label{the-gauge-invariant-uniqueness-theorem}Let $\Lambda $ be a
finitely aligned $k$-graph such that ${\mathcal{E}}\subseteq \operatorname{FE}(
\Lambda ^{1}) $ is efficient. Suppose that $\{ S_{\lambda }:\lambda \in
\Lambda \} $ is a relative Cuntz-Krieger $( \Lambda ;{\mathcal{E}}) $-family
in a $C^{\ast }$-algebra $B$ which satisfies:

\begin{enumerate}
\item[(G1)] $S_{v}\neq 0$ for all $v\in \Lambda ^{0}$;

\item[(G2)] $\prod_{e\in E}( S_{r( E) }-S_{e}S_{e}^{\ast }) \neq 0$ for all $%
E\in \operatorname{FE}( \Lambda ^{1}) \backslash \widehat{{\mathcal{E}}}$; and

\item[(G3)] there exists an action $\theta :\mathbb{T}^{k}\rightarrow
\operatorname{Aut}( B) $ with $\theta _{z}( S_{\lambda }) =z^{d( \lambda )
}S_{\lambda }$ for $z\in \mathbb{T}^{k}$,$\lambda \in \Lambda $.
\end{enumerate}

The homomorphism $\pi _{S}$ obtained from the universal property of $%
C^{\ast }(\Lambda ;{\mathcal{E}})$, is injective.
\end{theorem}
The only difference between our Theorem \ref{the-gauge-invariant-uniqueness-theorem} and Sims's gauge-invariant
uniqueness theorem \cite[Theorem 6.1]{Si06} is (G2), which in \cite{Si06},
(G2) is replaced by the following condition:%
\begin{equation}
\prod_{\lambda \in E}( S_{r( E) }-S_{\lambda }S_{\lambda }^{\ast }) \neq 0%
\text{ for all }E\in \operatorname{FE}( \Lambda ) \backslash \overline{{\mathcal{E}}}%
\text{.}  \label{equ-Sims-gauge-invariant-theorem}
\end{equation}%
Hence in order to show the two uniqueness theorem are identical, it suffices to
show that (G2) is equivalent to \eqref{equ-Sims-gauge-invariant-theorem}.

\begin{lemma}
\label{properties-of-rCK-efficient-and-satiated}Suppose that $\{ S_{\lambda
}:\lambda \in \Lambda \} $ is a relative Cuntz-Krieger $( \Lambda ;{\mathcal{%
E}}) $-family in a $C^{\ast }$-algebra $B$. Then the following two
conditions are equivalent:

\begin{enumerate}
\item[(a)] $\prod_{\lambda \in E}( S_{r( E) }-S_{\lambda }S_{\lambda }^{\ast
}) \neq 0$ for all $E\in \left. \operatorname{FE}( \Lambda ) \right\backslash
\overline{{\mathcal{E}}}$.

\item[(b)] $\prod_{e\in E}( S_{r( E) }-S_{e}S_{e}^{\ast }) \neq 0$ for all $%
E\in \operatorname{FE}( \Lambda ^{1}) \backslash \widehat{{\mathcal{E}}}$.
\end{enumerate}
\end{lemma}

\begin{proof}
To show (a)$\Rightarrow $(b), it suffices to show $\operatorname{FE}(\Lambda
^{1})\backslash \widehat{{\mathcal{E}}}\subseteq \operatorname{FE}(\Lambda
)\backslash \overline{{\mathcal{E}}}$. Take $E\in \operatorname{FE}(\Lambda
^{1})\backslash \widehat{{\mathcal{E}}}$. Then $E\in \operatorname{FE}(\Lambda )$.
Since $\widehat{{\mathcal{E}}}=\overline{{\mathcal{E}}}\cap \operatorname{FE}%
(\Lambda ^{1})$ (Proposition \ref{connect-efficient-to-satiated}), then $%
E\notin \widehat{{\mathcal{E}}}$ implies $E\notin \overline{{\mathcal{E}}}$.
Therefore $E\in \left. \operatorname{FE}(\Lambda )\right\backslash \overline{{%
\mathcal{E}}}$ and $\operatorname{FE}(\Lambda ^{1})\backslash \widehat{{\mathcal{E}}}%
\subseteq \operatorname{FE}(\Lambda )\backslash \overline{{\mathcal{E}}}$.

For (b)$\Rightarrow $(a), suppose for contradiction that there
exists $E\in \operatorname{FE}( \Lambda ) \backslash \overline{{\mathcal{E}}}$ such
that $\prod_{e\in E}( S_{r( E) }-S_{e}S_{e}^{\ast }) =0$. Consider the set ${%
\mathcal{E}}_{1}:=E\cup \overline{{\mathcal{E}}}$. Since $\{ S_{\lambda
}:\lambda \in \Lambda \} $ is a relative Cuntz-Krieger $( \Lambda ;{\mathcal{%
E}}) $-family and $\prod_{e\in E}( S_{r( E) }-S_{e}S_{e}^{\ast }) =0 $, then
for $F\in {\mathcal{E}}_{1}$, $\prod_{f\in F}(S_{r( F) }-S_{f}S_{f}^{\ast
})=0$. Following the argument of \cite[Corollary 5.6]{Si06}, this implies
\begin{equation}
\prod_{f\in F}(S_{r( F) }-S_{f}S_{f}^{\ast })=0\text{ for }F\in \overline{{%
\mathcal{E}}_{1}}.  \label{equ-F-in-epsilon}
\end{equation}%
Since $E\in \operatorname{FE}( \Lambda ) \backslash \overline{{\mathcal{E}}}$ and ${%
\mathcal{E}}_{1}=E\cup \overline{{\mathcal{E}}}$, we have $\overline{{%
\mathcal{E}}}\varsubsetneq {\mathcal{E}}_{1}$. So $\overline{{\mathcal{E}}}%
\varsubsetneq \overline{{\mathcal{E}}_{1}}$ since ${\mathcal{E}}%
_{1}\subseteq \overline{{\mathcal{E}}_{1}}$. By Proposition \ref%
{connect-efficient-to-satiated} and Corollary \ref%
{bijection-satiated-and-minimally-satiation}, $\widehat{{\mathcal{E}}}=(
\overline{{\mathcal{E}}}\cap \operatorname{FE}( \Lambda ^{1}) ) \varsubsetneq (
\overline{{\mathcal{E}}_{1}}\cap \operatorname{FE}( \Lambda ^{1}) ) $. Take $F\in
\overline{{\mathcal{E}}_{1}}\cap \operatorname{FE}( \Lambda ^{1}) $ with $F\notin
\widehat{{\mathcal{E}}}$. Since $F\in \overline{{\mathcal{E}}_{1}}$, by %
\eqref{equ-F-in-epsilon}, $\prod_{f\in F}(S_{r( F) }-S_{f}S_{f}^{\ast })=0$,
which contradicts Condition (b) since $F\in \operatorname{FE}( \Lambda ^{1})
\backslash \widehat{{\mathcal{E}}}$. The conclusion follows.
\end{proof}


An advantage of our version of the theorem is that
our condition (G2) is more checkable; there are fewer sets to consider.

\subsection{Gauge-invariant ideals in a relative Cuntz-Krieger algebra}

\label{Subsection-gauge-invariant-ideals}Another  application of
Theorem \ref{bijection-efficient-and-satiated} is to give a complete listing
of the gauge-invariant ideals in a relative Cuntz-Krieger algebra (Theorem %
\ref{gauge-invariant-ideals-efficient}). This is a simplification of Theorem
4.6 of \cite{SWW14} (see Remark \ref{gauge-invariant-ideals-efficient-remark}%
). First we give some preliminary notation and results.

Suppose that $\Lambda $ is a finitely aligned $k$-graph and that ${\mathcal{E%
}}\subseteq \operatorname{FE}(\Lambda ^{1})$. We define a relation $\geq $ on $%
\Lambda ^{0}$ by $w\geq v$ if and only if $v\Lambda w\neq \emptyset $. A
subset $H\subseteq \Lambda ^{0}$ is \emph{hereditary} if $v\in H$ and $w\geq
v$ imply $w\in H$; and $H$ is ${\mathcal{E}}$\emph{-saturated} if, whenever $%
v\in \Lambda ^{0}$ and $E\in v{\mathcal{E}}$ with $s(E)\subseteq H$, we have
$v\in H$. For a hereditary subset $H\subseteq \Lambda ^{0}$, the subcategory%
\begin{equation*}
\Lambda \backslash \Lambda H:=\{ \lambda \in \Lambda :s(\lambda )\notin H\}
\end{equation*}%
is a finitely aligned $k$-graph (see \cite[Lemma 4.1]{Si06G}).

\begin{lemma}
Suppose that $\Lambda $ is a finitely aligned $k$-graph, that $\mathcal{E}%
\subseteq \operatorname{FE}( \Lambda ^{1}) $ is an efficient set, and that $%
H\subseteq \Lambda ^{0}$ is a $\mathcal{E}$-saturated hereditary set. Then
\begin{equation*}
\mathcal{E}_{H}:=\{ E\backslash EH:E\in \mathcal{E}\}
\end{equation*}%
is a subset of $\operatorname{FE}(( \Lambda \backslash \Lambda H) ^{1})$.
\end{lemma}

\begin{proof}
Take $E\in \mathcal{E}_{H}$. Write $E=F\backslash FH$ with $F\in \mathcal{E}$%
. Since $F\subseteq r(E)\Lambda ^{1}$ and $F$ is finite, we have $E\subseteq
r(E)(\Lambda \backslash \Lambda H)^{1}$ and $E$ is finite.

To show that $E~$is exhaustive, take $\lambda \in r(E)(\Lambda \backslash
\Lambda H)$. If $\lambda \in F\Lambda $, then $s(\lambda )\notin H$ implies $%
\lambda \in E\Lambda $ and $\operatorname{Ext}_{\Lambda \backslash \Lambda
H}(\lambda ;E)\neq \emptyset $, as required. So suppose $\lambda \in
r(F)\Lambda \backslash F\Lambda $. Suppose for contradiction that $(\Lambda
\backslash \Lambda H)^{\min }(\lambda ,e)=\emptyset $ for every $e\in E$.
Hence $\Lambda ^{\min }(\lambda ,e)\subseteq \Lambda H\times \Lambda H$ for $%
e\in E$, so $\operatorname{Ext}_{\Lambda }(\lambda ;E)\subseteq \Lambda H$. Note
that $\operatorname{Ext}_{\Lambda }(\lambda ;F)=\operatorname{Ext}_{\Lambda }(\lambda
;E)\cup \operatorname{Ext}_{\Lambda }(\lambda ;FH)$ and since $H$ is hereditary, $%
\operatorname{Ext}_{\Lambda }(\lambda ;FH)\subseteq \Lambda H$. So $\operatorname{Ext}%
_{\Lambda }(\lambda ;F)\subseteq \Lambda H$ and $s(\operatorname{Ext}_{\Lambda
}(\lambda ;F))\subseteq H$. Because $\lambda \in r(F)\Lambda \backslash
F\Lambda $ and $F\in \mathcal{E}$, by (E2), there exists $F^{\prime }\in
\mathcal{E}$ with $F^{\prime }\subseteq \operatorname{Ext}_{\Lambda }(\lambda ;F)$%
. Because $H$ is $\mathcal{E}$-saturated, $s(F^{\prime })\subseteq s(%
\operatorname{Ext}_{\Lambda }(\lambda ;F))\subseteq H$ and $F^{\prime }\in
\mathcal{E}$, then $s(\lambda )=r(F^{\prime })\in H$, which contradicts $%
\lambda \in r(E)(\Lambda \backslash \Lambda H)$. So there exists $e\in E$
such that $(\Lambda \backslash \Lambda H)^{\min }(\lambda ,e)\neq \emptyset $%
. Thus $E$ is exhaustive and $E\in \operatorname{FE}((\Lambda \backslash \Lambda
H)^{1})$.
\end{proof}

Now we state our classification theorem.

\begin{theorem}
\label{gauge-invariant-ideals-efficient}Let $\Lambda $ be a
finitely aligned $k$-graph such that ${\mathcal{E}}\subseteq \operatorname{FE}%
(\Lambda ^{1})$ is an efficient set. Suppose that $\{s_{\lambda }^{{\mathcal{%
E}}}:\lambda \in \Lambda \}$ is the universal relative Cuntz-Krieger $%
(\Lambda ;{\mathcal{E}})$-family. For $H\subseteq \Lambda ^{0}$ and $%
\mathcal{B}\subseteq \operatorname{FE}((\Lambda \backslash \Lambda H)^{1})$, define $%
I_{H,\mathcal{B}}$ to be the ideal generated by
\begin{equation*}
\{s_{v}^{\mathcal{E}}:v\in H\}\cup \Big\{\prod_{e\in E}(s_{r(E)}^{\mathcal{E}%
}-s_{e}^{\mathcal{E}}s_{e}^{\mathcal{E}\ast }):E\in \mathcal{B}\Big\}\text{.}
\end{equation*}%
Then $\Phi :(H,\mathcal{B})\mapsto I_{H,\mathcal{B}}$ is a bijection between
the set of pairs $(H,\mathcal{B})$ such that $H$ is a $\mathcal{E}$%
-saturated hereditary set and $\mathcal{B}\subseteq \operatorname{FE}((\Lambda
\backslash \Lambda H)^{1})$ is an efficient set such that $\mathcal{E}%
_{H}\subseteq \widehat{\mathcal{B}}$, and the gauge-invariant ideals in $%
{\normalsize C}^{\ast }(\Lambda ;{\mathcal{E}})$. 

For any gauge-invariant
ideal $I$, define%
\begin{equation*}
H_{I}:=\{v\in \Lambda ^{0}:s_{v}^{\mathcal{E}}\in I\}\text{,}
\end{equation*}%
\begin{equation*}
\mathcal{B}_{I}:=\min \Big(\Big\{E\in \operatorname{FE}((\Lambda \backslash \Lambda
H_{I})^{1}):\prod_{e\in E}(s_{r(E)}^{\mathcal{E}}-s_{e}^{\mathcal{E}}s_{e}^{%
\mathcal{E}\ast })\in I\Big\}\Big)\text{.}
\end{equation*}%
Then the inverse of $\Phi$ is given by $I\mapsto (H_{I},\mathcal{B}_{I})$.
\end{theorem}

The rest of this subsection is devoted to proving Theorem \ref%
{gauge-invariant-ideals-efficient}.

\begin{lemma}
\label{gauge-invariant-ideal-lemma}Suppose that $I$ is an ideal in $%
{\normalsize C}^{\ast }(\Lambda ;{\mathcal{E}})$. Then, $H_{I}$ is $\mathcal{%
E}$-saturated hereditary and $\mathcal{B}_{I}$ is an efficient set with $%
\mathcal{E}_{H_{I}}\subseteq \widehat{\mathcal{B}_{I}}$.
\end{lemma}

\begin{proof}
To show $H_{I}$ is hereditary, take $v\in H_{I}$ and $w\in \Lambda ^{0}$
with $w\geq v$. Take $\lambda \in v\Lambda w$. So%
\begin{equation*}
v\in H_{I}\Rightarrow s_{v}^{\mathcal{E}}\in I\Rightarrow s_{\lambda }^{%
\mathcal{E}}=s_{v}^{\mathcal{E}}s_{\lambda }^{\mathcal{E}}\in I\Rightarrow
s_{w}^{\mathcal{E}}=s_{\lambda }^{\mathcal{E\ast }}s_{\lambda }^{\mathcal{E}%
}\in I\Rightarrow w\in H_{I}
\end{equation*}%
and $H_{I}$ is hereditary.

To show that $H_{I}$ is $\mathcal{E}$-saturated, take $v\in \Lambda ^{0}$
and $E\in v{\mathcal{E}}$ with $s(E)\subseteq H_{I}$. For $e\in E$,
\begin{equation*}
s(e)\in H_{I}\Rightarrow s_{s(e)}^{\mathcal{E}}\in I\Rightarrow s_{e}^{%
\mathcal{E}}=s_{e}^{\mathcal{E}}s_{s(e)}^{\mathcal{E}}\in I\Rightarrow
s_{e}^{\mathcal{E}}s_{e}^{\mathcal{E\ast }}\in I\text{.}
\end{equation*}%
Since $E\in {\mathcal{E}}$ and $\{s_{\lambda }^{{\mathcal{E}}}s_{\lambda }^{{%
\mathcal{E\ast }}}:\lambda \in \Lambda \}$ is a commuting family (see \cite[%
Lemma 2.7(i)]{RSY04}),
\begin{equation*}
0=\prod_{e\in E}(s_{v}^{\mathcal{E}}-s_{e}^{\mathcal{E}}s_{e}^{\mathcal{%
E\ast }})=s_{v}^{\mathcal{E}}+\sum_{F\subseteq E,F\neq \emptyset
}(-1)^{\left\vert F\right\vert }\prod_{e\in F}(s_{e}^{\mathcal{E}}s_{e}^{%
\mathcal{E\ast }})
\end{equation*}%
and since $s_{e}^{\mathcal{E}}s_{e}^{\mathcal{E\ast }}\in I$ for each $e\in
E $, $s_{v}^{\mathcal{E}}\in I$. Therefore $v\in H_{I}$ and $H_{I}$ is $%
\mathcal{E}$-saturated.

To show that $\mathcal{B}_{I}$ is an efficient set with $\mathcal{E}%
_{H_{I}}\subseteq \widehat{\mathcal{B}_{I}}$, consider the quotient map $q:%
{\normalsize C}^{\ast }(\Lambda ;{\mathcal{E}})\rightarrow {\normalsize C}%
^{\ast }(\Lambda ;{\mathcal{E}})/I$. We show that $\mathcal{B}_{I}$ is
efficient. Write $\Gamma :=\Lambda \backslash \Lambda H_{I}$. (E1) follows
from definition of $\mathcal{B}_{I}$. For (E2$^{\prime }$), take $E\in
\mathcal{B}_{I}$ and $f\in r(E)\Gamma ^{1}\backslash E$. So $q(\prod_{e\in
E}(s_{r(E)}^{\mathcal{E}}-s_{e}^{\mathcal{E}}s_{e}^{\mathcal{E\ast }}))=0$.
By Proposition \ref{efficient-exhaustive-edge}, 
\[q(\prod_{g\in \operatorname{Ext}%
_{\Gamma }(f;E)}(s_{s(f)}^{\mathcal{E}}-s_{g}^{\mathcal{E}}s_{g}^{\mathcal{%
E\ast }}))=0.\] So $\operatorname{Ext}_{\Gamma }(f;E)\in \widehat{\mathcal{B}_{I}}$
and (E2$^{\prime }$) holds. 

For (E3$^{\prime }$), take $E\in
\mathcal{B}_{I}$, $e\in E$, and $F\in s(e)\mathcal{B}_{I}$. Define%
\begin{equation*}
E_{F}:=(E\backslash \{ e\} )\cup \{ (ef)(0,d(f)):f\in F\} \text{.}
\end{equation*}%
Note that 
\[q(\prod_{f\in E}(s_{r(E)}^{\mathcal{E}}-s_{f}^{\mathcal{E}}s_{f}^{%
\mathcal{E\ast }}))=0=q(\prod_{f\in F}(s_{s(E)}^{\mathcal{E}}-s_{f}^{%
\mathcal{E}}s_{f}^{\mathcal{E\ast }})).\] 
By Proposition \ref%
{properties-of-ideal}, 
\[q(\prod_{f\in E_{F}}(s_{r(E)}^{\mathcal{E}}-s_{f}^{%
\mathcal{E}}s_{f}^{\mathcal{E\ast }}))=0.\] So $E_{F}\in \widehat{\mathcal{B}%
_{I}}$. Thus (E3$^{\prime }$) holds and $\mathcal{B}_{I}$ is efficient.

We show that $\widehat{\mathcal{B}_{I}}$ contains $\mathcal{E}_{H_{I}}$.
Take $E\in \mathcal{E}_{H_{I}}$, we show $q(\prod_{e\in E}(s_{r(E)}^{%
\mathcal{E}}-s_{e}^{\mathcal{E}}s_{e}^{\mathcal{E\ast }}))=0$. Since $E\in
\mathcal{E}_{H_{I}}$, there exists $E^{\prime }\in \mathcal{E}$ such that $%
E=E^{\prime }\backslash E^{\prime }H_{I}$. Then $E^{\prime }=E\cup E^{\prime
}H_{I}$. For $e\in E^{\prime }H_{I}$, we have $s_{s(e)}^{\mathcal{E}}\in I$
and $s_{e}^{\mathcal{E}}s_{e}^{\mathcal{E}\ast }\in I$. So $q(\prod_{e\in
E^{\prime }H_{I}}(s_{r(E)}^{\mathcal{E}}-s_{e}^{\mathcal{E}}s_{e}^{\mathcal{%
E\ast }}))=q(s_{r(E)}^{\mathcal{E}})$ and%
\begin{align*}
q\Big(\prod_{e\in E}(s_{r(E)}^{\mathcal{E}}-s_{e}^{\mathcal{E}}s_{e}^{%
\mathcal{E\ast }})\Big)& =q\Big(\prod_{e\in E}(s_{r(E)}^{\mathcal{E}}-s_{e}^{%
\mathcal{E}}s_{e}^{\mathcal{E\ast }})\Big)q(s_{r(E)}^{\mathcal{E}}) \\
& =q\Big(\prod_{e\in E}(s_{r(E)}^{\mathcal{E}}-s_{e}^{\mathcal{E}}s_{e}^{%
\mathcal{E\ast }})\prod_{e\in E^{\prime }H_{I}}(s_{r(E)}^{\mathcal{E}%
}-s_{e}^{\mathcal{E}}s_{e}^{\mathcal{E\ast }})\Big) \\
& =q\Big(\prod_{e\in E^{\prime }}(s_{r(E)}^{\mathcal{E}}-s_{e}^{\mathcal{E}%
}s_{e}^{\mathcal{E\ast }})\Big)=q(0)=0
\end{align*}%
since $E^{\prime }\in \mathcal{E}$. So $E\in \widehat{\mathcal{B}_{I}}$ and $%
\mathcal{E}_{H_{I}}\subseteq \widehat{\mathcal{B}_{I}}$.
\end{proof}

\begin{proof}[Proof of Theorem \protect\ref{gauge-invariant-ideals-efficient}%
]
We use a similar argument to \cite[Theorem 4.9]{CBMS}. Write $\Gamma
:=\Lambda \backslash \Lambda H$. To show that $\Phi $ is surjective, take a
gauge-invariant ideal $I$ in ${\normalsize C}^{\ast }(\Lambda ;{\mathcal{E}}%
) $. By Lemma \ref{gauge-invariant-ideal-lemma}, $H_{I}$ is $\mathcal{E}$%
-saturated hereditary and $\mathcal{B}_{I}$ is efficient such that $\mathcal{%
E}_{H_{I}}\subseteq \widehat{\mathcal{B}_{I}}$. We show $I=I_{H_{I},\mathcal{%
B}_{I}}$. Since all the generators of $I_{H_{I},\mathcal{B}_{I}}$ are in $I$%
, we have $I_{H_{I},\mathcal{B}_{I}}\subseteq I$. To show the reverse
inclusion, consider the quotient maps%
\begin{equation*}
q_{I}:{\normalsize C}^{\ast }(\Lambda ;{\mathcal{E}})\rightarrow
{\normalsize C}^{\ast }(\Lambda ;{\mathcal{E}})/I\text{, }q_{I_{H_{I},%
\mathcal{B}_{I}}}:{\normalsize C}^{\ast }(\Lambda ;{\mathcal{E}})\rightarrow
{\normalsize C}^{\ast }(\Lambda ;{\mathcal{E}})/I_{H_{I},\mathcal{B}_{I}}%
\text{, and}
\end{equation*}%
\begin{equation*}
q_{I/I_{H_{I},\mathcal{B}_{I}}}:{\normalsize C}^{\ast }(\Lambda ;{\mathcal{E}%
})/I_{H_{I},\mathcal{B}_{I}}\rightarrow {\normalsize C}^{\ast }(\Lambda ;{%
\mathcal{E}})/I=({\normalsize C}^{\ast }(\Lambda ;{\mathcal{E}})/I_{H_{I},%
\mathcal{B}_{I}})/(I/I_{H_{I},\mathcal{B}_{I}}).
\end{equation*}%
So $q_{I}=q_{I/I_{H_{I},\mathcal{B}_{I}}}\circ q_{I_{H_{I},\mathcal{B}_{I}}}$%
. We claim both $\{q_{I}(s_{\lambda }^{{\mathcal{E}}}):\lambda \in \Gamma \}$
and $\{q_{I_{H_{I},\mathcal{B}_{I}}}(s_{\lambda }^{{\mathcal{E}}}):\lambda
\in \Gamma \}$ are relative Cuntz-Krieger $(\Gamma ;\mathcal{B}_{I})$%
-families. Since $q_{I},q_{I_{H_{I},\mathcal{B}_{I}}}$ are quotient maps and
$q_{I}(s_{\lambda }^{{\mathcal{E}}})=0=q_{I_{H_{I},\mathcal{B}%
_{I}}}(s_{\lambda }^{{\mathcal{E}}})$ for $\lambda \in \Lambda H$, both $%
\{q_{I}(s_{\lambda }^{{\mathcal{E}}}):\lambda \in \Gamma \}$ and $%
\{q_{I_{H_{I},\mathcal{B}_{I}}}(s_{\lambda }^{{\mathcal{E}}}):\lambda \in
\Gamma \}$ are Toeplitz-Cuntz-Krieger $\Gamma $-families. For $E\in \mathcal{%
B}_{I}$, based on definition of $\mathcal{B}_{I}$ and $I_{H_{I},\mathcal{B}%
_{I}}$, we have
\begin{equation*}
q_{I}\Big(\prod_{e\in E}(s_{r(E)}^{\mathcal{E}}-s_{e}^{\mathcal{E}}s_{e}^{%
\mathcal{E\ast }})\Big)=0=q_{I_{H_{I},\mathcal{B}_{I}}}\Big(\prod_{e\in
E}(s_{r(E)}^{\mathcal{E}}-s_{e}^{\mathcal{E}}s_{e}^{\mathcal{E\ast }})\Big)%
\text{.}
\end{equation*}%
So both families satisfy (CK) and are relative Cuntz-Krieger $(\Gamma ;%
\mathcal{B}_{I})$-families, as claimed.

Let $\pi _{H_{I},\mathcal{B}_{I}}:{\normalsize C}^{\ast }(\Gamma ;\mathcal{B}%
_{I})\rightarrow {\normalsize C}^{\ast }(\Lambda ;\mathcal{E})/I_{H_{I},%
\mathcal{B}_{I}}$ and $\pi _{I}:{\normalsize C}^{\ast }(\Gamma ;\mathcal{B}%
_{I})\rightarrow {\normalsize C}^{\ast }(\Lambda ;\mathcal{E})/I$ be the
homomorphisms obtained from the universal property of ${\normalsize C}^{\ast
}(\Gamma ;\mathcal{B}_{I})$. So $\pi _{I}$ and $q_{I_{H_{I},\mathcal{B}%
_{I}}}\circ \pi _{H_{I},\mathcal{B}_{I}}$ are homomorphisms which agree on $%
\{ s_{\lambda }^{{\mathcal{E}}}s_{\mu }^{{\mathcal{E\ast }}}:\lambda ,\mu
\in \Gamma \} $, which is the generators of ${\normalsize C}^{\ast }(\Gamma ;%
\mathcal{B}_{I})$ (see \cite[Lemma 2.7(iv)]{RSY04}), and hence are equal. We
use Theorem \ref{the-gauge-invariant-uniqueness-theorem} to show that $\pi
_{I}$ is injective. By definition of $H_{I}$, $q_{I}(s_{v}^{\mathcal{E}%
})\neq 0$ for all $v\in \Gamma ^{0}$, so $\{q_{I}(s_{\lambda }^{{\mathcal{E}}%
}):\lambda \in \Gamma \}$ satisfies (G1). If $E\in \operatorname{FE}(\Gamma ^{1})$
such that $E\notin \widehat{\mathcal{B}_{I}} $, then 
\[q_{I}(\prod_{e\in E}(s_{r(E)}^{\mathcal{E}}-s_{e}^{\mathcal{E}}s_{e^{\ast }}^{\mathcal{E}
}))\neq 0\] 
by definition of $\mathcal{B}_{I}$. So $\{q_{I}(s_{\lambda }^{{%
\mathcal{E}}}):\lambda \in \Gamma \}$ satisfies (G2). Since $I$ is
gauge-invariant, then the gauge action $\beta $ on ${\normalsize C}^{\ast
}(\Lambda ;\mathcal{E})$ descends to an action $\theta $ on ${\normalsize C}%
^{\ast }(\Lambda ;\mathcal{E})/I$ which satisfies (G3). Therefore, by
Theorem \ref{the-gauge-invariant-uniqueness-theorem}, $\pi _{I} $ is
injective. Since $\pi _{H_{I},\mathcal{B}_{I}}$ is surjective and $\pi
_{I}=q_{I/I_{H_{I},\mathcal{B}_{I}}}\circ \pi _{H_{I},\mathcal{B}_{I}}$, the
injectivity of $\pi _{I}$ implies that $q_{I/I_{H_{I},\mathcal{B}_{I}}}$ is
also injective. Thus $I_{H_{I},\mathcal{B}_{I}}=I$, as required.

Next we show the injectivity of $\Phi $. Take a $\mathcal{E}$-saturated
hereditary set $H$ and an efficient set $\mathcal{B}\subseteq \operatorname{FE}%
(\Gamma ^{1})$ such that $\mathcal{E}_{H}\subseteq \widehat{\mathcal{B}}$.
Define
\begin{equation*}
H_{I_{H,\mathcal{B}}}:=\{v\in \Lambda ^{0}:s_{v}^{\mathcal{E}}\in I_{H,%
\mathcal{B}}\}\text{,}
\end{equation*}%
\begin{equation*}
\mathcal{B}_{I_{H,\mathcal{B}}}:=\min \Big(\Big\{E\in \operatorname{FE}(\Gamma
^{1}):\prod_{e\in E}(s_{r(E)}^{\mathcal{E}}-s_{e}^{\mathcal{E}}s_{e}^{%
\mathcal{E\ast }})\in I_{H,\mathcal{B}}\Big\}\Big)\text{.}
\end{equation*}%
To show that $\Phi $ is injective, we show $H=H_{I_{H,\mathcal{B}}}$ and $%
\mathcal{B}=\mathcal{B}_{I_{H,\mathcal{B}}}$. We trivially have $H\subseteq
H_{1}$ and $\widehat{\mathcal{B}}\subseteq \widehat{\mathcal{B}_{I_{H,%
\mathcal{B}}}}$. To prove the reverse inclusion, consider the universal
relative Cuntz-Krieger $(\Gamma ;\mathcal{B})$-family $\{s_{\lambda }^{%
\mathcal{B}}:\lambda \in \Gamma \}$. For $\lambda \in \Lambda $, we define%
\begin{equation*}
S_{\lambda }:=%
\begin{cases}
s_{\lambda }^{\mathcal{B}} & \text{if }s(\lambda )\notin H\text{,} \\
0 & \text{if }s(\lambda )\in H\text{.}%
\end{cases}%
\end{equation*}%
We show that $\{S_{\lambda }:\lambda \in \Lambda \}$ is a relative
Cuntz-Krieger $(\Lambda ;\mathcal{E})$-family. It is clear that the family
satisfies (TCK1-3). To show (CK), take $E\in \mathcal{E}$. Then%
\begin{align*}
\prod_{e\in E}(S_{r(E)}-S_{e}S_{e}^{\ast })& =\prod_{e\in
EH}(S_{r(E)}-S_{e}S_{e}^{\ast })\prod_{e\in E\backslash
EH}(S_{r(E)}-S_{e}S_{e}^{\ast }) \\
& =S_{r(E)}\prod_{e\in E\backslash EH}(S_{r(E)}-S_{e}S_{e}^{\ast
})=s_{r(E)}^{\mathcal{B}}\prod_{e\in E\backslash EH}(s_{r(E)}^{\mathcal{B}%
}-s_{e}^{\mathcal{B}}s_{e}^{\mathcal{B\ast }})=0
\end{align*}%
since $E\backslash EH\in \mathcal{E}_{H}\subseteq \widehat{\mathcal{B}}$ and
Proposition \ref{cRK-family-without-efficient}. So $\{S_{\lambda }:\lambda
\in \Lambda \}$ is a relative Cuntz-Krieger $(\Lambda ;\mathcal{E})$-family
and the universal property of $C^{\ast }(\Lambda ;\mathcal{E})$ gives a
homomorphism $\pi _{S}$ of $C^{\ast }(\Lambda ;\mathcal{E})$ into $C^{\ast
}(\Gamma ;\mathcal{B})$ with $\pi _{S}(s_{\lambda }^{\mathcal{E}%
})=S_{\lambda }$ for $\lambda \in \Lambda $. Since $\pi _{S}(s_{v}^{\mathcal{%
E}})=0$ for $v\in H$ and $\pi _{S}(\prod_{e\in E}(s_{r(E)}^{\mathcal{E}%
}-s_{e}^{\mathcal{E}}s_{e}^{\mathcal{E\ast }}))=0$ for $E\in \mathcal{B}$,
we have $I_{H,\mathcal{B}}\subseteq \ker \pi _{S}$. By Proposition \ref%
{properties-of-relative-CK-algebra},
\begin{equation*}
v\notin H\Rightarrow s_{v}^{\mathcal{B}}\neq 0\Rightarrow S_{v}\neq
0\Rightarrow \pi _{S}(s_{v}^{\mathcal{E}})\neq 0\Rightarrow s_{v}^{\mathcal{E%
}}\notin \ker \pi _{S}\Rightarrow s_{v}^{\mathcal{E}}\notin I_{H,\mathcal{B}}%
\text{,}
\end{equation*}%
\begin{align*}
E\in \operatorname{FE}(\Gamma ^{1})\backslash \widehat{\mathcal{B}}& \Rightarrow
\prod_{e\in E}(s_{r(E)}^{\mathcal{B}}-s_{e}^{\mathcal{B}}s_{e}^{\mathcal{%
B\ast }})\neq 0\Rightarrow \prod_{e\in E}(S_{r(E)}-S_{e}S_{e}^{\ast })\neq 0
\\
& \Rightarrow \pi _{S}\Big(\prod_{e\in E}(s_{r(E)}^{\mathcal{E}}-s_{e}^{%
\mathcal{E}}s_{e}^{\mathcal{E\ast }})\Big)\neq 0\Rightarrow \prod_{e\in
E}(s_{r(E)}^{\mathcal{E}}-s_{e}^{\mathcal{E}}s_{e}^{\mathcal{E\ast }})\notin
\ker \pi _{S} \\
& \Rightarrow \prod_{e\in E}(s_{r(E)}^{\mathcal{E}}-s_{e}^{\mathcal{E}%
}s_{e}^{\mathcal{E\ast }})\notin I_{H,\mathcal{B}}\text{.}
\end{align*}%
Then $H_{I_{H,\mathcal{B}}}\subseteq H$ and $\widehat{\mathcal{B}_{I_{H,%
\mathcal{B}}}}\subseteq \widehat{\mathcal{B}}$. Hence $H=H_{I_{H,\mathcal{B}%
}}$ and $\widehat{\mathcal{B}}=\widehat{\mathcal{B}_{I_{H,\mathcal{B}}}}$.
By Lemma \ref{efficient-min-e-hat}, $\mathcal{B}=\mathcal{B}_{I_{H,%
\mathcal{B}}}$. Therefore $\Phi $ is injective and hence an isomorphism.
\end{proof}

\begin{remark}
\label{gauge-invariant-ideals-efficient-additional}From the proof of
Theorem \ref{gauge-invariant-ideals-efficient}, we get:

\begin{enumerate}
\item[(a)] For a gauge-invariant ideal $I$, we have $I_{H_{I},\mathcal{B}%
_{I}}=I$.

\item[(b)] For a hereditary set $H$ and an efficient set $\mathcal{B}%
\subseteq \operatorname{FE}((\Lambda \backslash \Lambda H)^{1})$ such that $\mathcal{%
E}_{H}\subseteq \widehat{\mathcal{B}}$, we have $H_{I_{H,\mathcal{B}}}=H$
and $\mathcal{B}_{I_{H,\mathcal{B}}}=\mathcal{B}$.
\end{enumerate}

Part (a)\ is from the surjectivity of \ $\Phi $. The injectivity of $\Phi $
implies part (b).
\end{remark}

\begin{remark}
\label{gauge-invariant-ideals-efficient-remark}In \cite{SWW14}, Sims,
Whitehead and Whittaker gives a complete listing of the gauge-invariant
ideals in a twisted $C^{\ast }$-algebra associated to a higher-rank graph
\cite[Theorem 4.6]{SWW14}. Since
their twisted $C^{\ast }$-algebras can be viewed as a generalisation of relative
Cuntz-Krieger algebras, they actually have an alternative version of Theorem %
\ref{gauge-invariant-ideals-efficient}, which uses satiated sets rather than
efficient sets. Indeed, we could have shown Theorem \ref%
{gauge-invariant-ideals-efficient} as a consequence of Theorem 4.6 of \cite%
{SWW14}. However, the direct argument above takes about the same amount of
effort.
\end{remark}

\subsection{Toeplitz algebras and their quotient algebras}

\label{Subsection-Toeplitz-algebras}Throughout this subsection, suppose that $%
\Lambda $ is a row-finite $k$-graph with no sources.
In this subsection, we study the relationship between Toeplitz algebras and
their ideals and quotient algebras. 

Historically, one was forced to consider satiations whenever working with 
ideals and quotients of higher-rank graph Toeplitz algebras.  
For example in \cite[Appendix A]{aHKR15}, an Huef, Kang
and Raeburn must prove results about satiations even though 
they are really only interested in $\mathcal{%
E}:=\{ \bigcup_{i\in K}v\Lambda ^{e_{i}}:v\in \Lambda ^{0}\} $ where $K$ is
a nonempty subset of $\{ 1,\ldots ,k\} $. 
Remark \ref{example-efficient} tells us that $\{ \bigcup_{i\in
K}v\Lambda ^{e_{i}}:v\in \Lambda ^{0}\} $ is efficient and so we have established 
tools that allow us to avoid these unruly satiations.

Our next theorem is a special
case of Theorem \ref{gauge-invariant-ideals-efficient}, which lists all the
gauge-invariant ideals in a higher-rank graph Toeplitz algebra. Here we adjust some
notation as explained in Remark \ref{gauge-invariant-ideals-Toeplitz-algebra-remark}. 
\begin{theorem}
\label{gauge-invariant-ideals-Toeplitz-algebra}Suppose that $\Lambda $ is a
row-finite $k$-graph with no sources and that $\{t_{\lambda }:\lambda \in
\Lambda \}$ is the universal Toeplitz-Cuntz-Krieger $\Lambda $-family.

\begin{enumerate}
\item[(a)] Suppose that $H\subseteq \Lambda ^{0}$ and that $\mathcal{E}%
\subseteq \operatorname{FE}((\Lambda \backslash \Lambda H)^{1})$. The ideal $I_{H,%
\mathcal{E}}$, as defined in Theorem \ref{gauge-invariant-ideals-efficient},
is a gauge-invariant ideal in $TC^{\ast }(\Lambda )$.

\item[(b)] Suppose that $I$ is a gauge-invariant ideal of $TC^{\ast
}(\Lambda )$. Suppose that $H_{I}$ and $\mathcal{E}_{I}$ are as defined in
Theorem \ref{gauge-invariant-ideals-efficient}. Then $H_{I}$ is a hereditary
subset of $\Lambda $, and $\mathcal{E}_{I}$ is an efficient subset of $\operatorname{%
FE}((\Lambda \backslash \Lambda H_{I})^{1})$.

\item[(c)] Suppose that $I$ is a gauge-invariant ideal of $TC^{\ast
}(\Lambda )$. If $E\in \operatorname{FE}((\Lambda \backslash \Lambda H_{I})^{1})\,$\
such that $\prod_{e\in E}(t_{r(E)}-t_{e}t_{e}^{\ast })\in I$, then there
exists $F\in \mathcal{E}_{I}$ such that $F\subseteq E$.

\item[(d)] For a gauge-invariant ideal $I$, we have $I_{H_{I},\mathcal{E}%
_{I}}=I$.

\item[(e)] For a hereditary set $H$ of $\Lambda $ and an efficient set $%
\mathcal{E}\subseteq \operatorname{FE}((\Lambda \backslash \Lambda H)^{1})$, we have
$H_{I_{H,\mathcal{E}}}=H$ and $\mathcal{E}_{I_{H,\mathcal{E}}}=\mathcal{E}$.
\end{enumerate}
\end{theorem}

\begin{remark}
\label{gauge-invariant-ideals-Toeplitz-algebra-remark}A number of aspects
are worth commenting on:

\begin{enumerate}
\item[(i)] Since Toeplitz-Cuntz-Krieger $\Lambda $-families coincide with
relative Cuntz-Krieger $( \Lambda ;{\mathcal{\emptyset }}) $-families, then
to simplify the notation, we replace  $\mathcal{B}$ of
Theorem \ref{gauge-invariant-ideals-efficient} with $\mathcal{E}$.

\item[(ii)] Part (c) holds since $\mathcal{E}_{I}$ is efficient. Part (d)
and (e)\ follows from Remark \ref%
{gauge-invariant-ideals-efficient-additional}.
\end{enumerate}
\end{remark}

Here for $B\subseteq TC^{\ast }(\Lambda
) $, we write $\left\langle B\right\rangle $ to denote the ideal of $%
TC^{\ast }(\Lambda )$ generated by the elements of $B$. Note that $%
\left\langle B\right\rangle $ is the smallest ideal which contains $B$.

\begin{proposition}
\label{diagram-Toeplitz-algebra}Suppose that $\Lambda $ is a row-finite $k$%
-graph with no sources and that $\{t_{\lambda }:\lambda \in \Lambda \}$ is
the universal Toeplitz-Cuntz-Krieger $\Lambda $-family. Suppose that $K,L\ $%
are nonempty subsets of $\{ 1,\ldots ,k\} $. Define%
\begin{equation*}
A:=\Big\{\prod_{i\in K}\Big(t_{v}-\sum_{e\in v\Lambda
^{e_{i}}}t_{e}t_{e}^{\ast }\Big):v\in \Lambda ^{0}\Big\}\text{ and }B:=\Big\{%
\prod_{j\in L}\Big(t_{v}-\sum_{e\in v\Lambda ^{e_{j}}}t_{e}t_{e}^{\ast }\Big)%
:v\in \Lambda ^{0}\Big\}\text{.}
\end{equation*}%
Then the diagram
\begin{equation*}
\begin{array}{ccccccccc}
&  & 0 &  & 0 &  & 0 &  &  \\
&  & \downarrow &  & \downarrow &  & \downarrow &  &  \\
0 & \rightarrow & \left\langle AB\right\rangle & \rightarrow & \left\langle
A\right\rangle & \rightarrow & \frac{\left\langle A\cup B\right\rangle }{%
\left\langle B\right\rangle } & \rightarrow & 0 \\
&  & \downarrow &  & \downarrow &  & \downarrow &  &  \\
0 & \rightarrow & \left\langle B\right\rangle & \rightarrow & TC^{\ast
}(\Lambda ) & \rightarrow & \frac{TC^{\ast }(\Lambda )}{\left\langle
B\right\rangle } & \rightarrow & 0 \\
&  & \downarrow &  & \downarrow &  & \downarrow &  &  \\
0 & \rightarrow & \frac{\left\langle A\cup B\right\rangle }{\left\langle
A\right\rangle } & \rightarrow & \frac{TC^{\ast }(\Lambda )}{\left\langle
A\right\rangle } & \rightarrow & \frac{TC^{\ast }(\Lambda )}{\left\langle
A\cup B\right\rangle } & \rightarrow & 0 \\
&  & \downarrow &  & \downarrow &  & \downarrow &  &  \\
&  & 0 &  & 0 &  & 0 &  &
\end{array}%
\end{equation*}%
is commutative and all the rows and columns are exact.
\end{proposition}

Before giving the proof, we establish a 
stepping stone result.

\begin{lemma}
\label{ideals-in-Toeplitz-algebra}Suppose that $K,L\ $are nonempty subsets
of $\{ 1,\ldots ,k\} $ and that $A,B$ are as in Theorem \ref%
{diagram-Toeplitz-algebra}. Then $\left\langle A\right\rangle +\left\langle
B\right\rangle =\left\langle A\cup B\right\rangle $ and $\left\langle
A\right\rangle \cap \left\langle B\right\rangle =\langle AB\rangle =\langle
BA\rangle $.
\end{lemma}

\begin{proof}
We have $A\subseteq \left\langle A\cup B\right\rangle $ and $B\subseteq
\left\langle A\cup B\right\rangle $. Because $\left\langle A\right\rangle $
is the smallest ideal which contains $A$, $\left\langle A\right\rangle
\subseteq \left\langle A\cup B\right\rangle $ and similarly, $\left\langle
B\right\rangle \subseteq \left\langle A\cup B\right\rangle $. So $%
\left\langle A\right\rangle +\left\langle B\right\rangle \subseteq
\left\langle A\cup B\right\rangle $.

We show the reverse inclusion. Since $A\subseteq \left\langle A\right\rangle
+\left\langle B\right\rangle $ and $B\subseteq \left\langle A\right\rangle
+\left\langle B\right\rangle $, we have $A\cup B\subseteq \left\langle
A\right\rangle +\left\langle B\right\rangle $. So $\left\langle A\cup
B\right\rangle \subseteq \left\langle A\right\rangle +\left\langle
B\right\rangle $ since $\left\langle A\right\rangle +\left\langle
B\right\rangle $ is a closed ideal, as required.

Now we show $\left\langle A\right\rangle \cap \left\langle B\right\rangle
=\langle AB\rangle $. Write $I:=\left\langle A\right\rangle \cap
\left\langle B\right\rangle $. First we show that $I$ is a gauge-invariant
ideal. Define $\mathcal{E}:=\{ \bigcup_{i\in K}v\Lambda ^{e_{i}}:v\in
\Lambda ^{0}\} $ and we have
\begin{equation*}
\left\langle A\right\rangle =\Big<\Big\{\prod_{i\in K}\prod_{e\in v\Lambda
^{e_{i}}}(t_{v}-t_{e}t_{e}^{\ast }):v\in \Lambda ^{0}\Big\}\Big>=\Big<\Big\{%
\prod_{e\in E}(t_{r(E)}-t_{e}t_{e}^{\ast }):E\in \mathcal{E}\Big\}\Big>.
\end{equation*}%
By Remark \ref{example-efficient}, $\mathcal{E}$ is efficient and by Theorem %
\ref{gauge-invariant-ideals-Toeplitz-algebra}(a), we have
\begin{equation}
\left\langle A\right\rangle =I_{\emptyset ,\mathcal{E}}
\label{equ-ideals-Toeplitz-<A>}
\end{equation}%
is a gauge-invariant ideal. Similarly, $\left\langle B\right\rangle $ is a
gauge-invariant ideal. For $z\in \mathbb{T}^{k}$, we have $\gamma
_{z}(\left\langle A\right\rangle )\subseteq \left\langle A\right\rangle $
and $\gamma _{z}(\left\langle B\right\rangle )\subseteq \left\langle
B\right\rangle $ where $\gamma $ is the gauge action of $\mathbb{T}^{k}$ on $%
TC^{\ast }(\Lambda )$. Then%
\begin{equation*}
\gamma _{z}(I)\subseteq \gamma _{z}(\left\langle A\right\rangle )\subseteq
\left\langle A\right\rangle \text{ and }\gamma _{z}(I)\subseteq \gamma
_{z}(\left\langle B\right\rangle )\subseteq \left\langle B\right\rangle
\text{.}
\end{equation*}%
So $\gamma _{z}(I)\subseteq \left\langle A\right\rangle \cap \left\langle
B\right\rangle =I$ and $I$ is a gauge-invariant ideal.

Now we investigate $H_{I}$ and $\mathcal{E}_{I}$. For $v\in \Lambda ^{0}$,
we have
\begin{equation}
\prod_{i\in K\cup L}\prod_{e\in v\Lambda ^{e_{i}}}(t_{v}-t_{e}t_{e}^{\ast
})\in \left\langle A\right\rangle \cap \left\langle B\right\rangle =I\text{.}
\label{equ-ideals-Toeplitz-belong-to-I}
\end{equation}%
Thus $I\neq 0$. Since $I$ is a gauge invariant ideal, by Theorem \ref%
{gauge-invariant-ideals-Toeplitz-algebra}(b), $H_{I}$ is hereditary and $%
\mathcal{E}_{I}$ is efficient$\,$. By Theorem \ref%
{gauge-invariant-ideals-Toeplitz-algebra}(e), \eqref{equ-ideals-Toeplitz-<A>}
implies $H_{I_{\emptyset ,\mathcal{E}}}=\emptyset $ and $H_{\left\langle
A\right\rangle }=\emptyset $. Thus for $v\in \Lambda ^{0}$, $t_{v}\notin
\left\langle A\right\rangle $ and since $I\subseteq \left\langle
A\right\rangle $, $t_{v}\notin I$. \ Hence
\begin{equation}
H_{I}=\emptyset .  \label{equ-ideals-Toeplitz-H_I}
\end{equation}

We claim $\mathcal{E}_{I}=\{\bigcup_{i\in K\cup L}v\Lambda ^{e_{i}}:v\in
\Lambda ^{0}\}$. To show $\mathcal{E}_{I}\subseteq \{\bigcup_{i\in K\cup
L}v\Lambda ^{e_{i}}:v\in \Lambda ^{0}\}$, take $E\in \mathcal{E}_{I}$. Write
$G:=\bigcup_{i\in K\cup L}r( E) \Lambda ^{e_{i}}$. We claim $E=G$. We have $%
\prod_{e\in E}(t_{r( E) }-t_{e}t_{e}^{\ast })\in I\subseteq \left\langle
A\right\rangle =I_{\emptyset ,\mathcal{E}}$ (see %
\eqref{equ-ideals-Toeplitz-<A>}). Since $E\in \operatorname{FE}( \Lambda ^{1}) $, by
Theorem \ref{gauge-invariant-ideals-Toeplitz-algebra}(c), $\bigcup_{i\in
K}r( E) \Lambda ^{e_{i}}\subseteq E$. Using a similar argument to $%
\left\langle B\right\rangle $, $\bigcup_{j\in L}r( E) \Lambda
^{e_{j}}\subseteq E$. Hence $G\subseteq E$. Since $\prod_{e\in G}(t_{r( E)
}-t_{e}t_{e}^{\ast })\in I$ (see \eqref{equ-ideals-Toeplitz-belong-to-I}),
by Theorem \ref{gauge-invariant-ideals-Toeplitz-algebra}(c), there exists $%
F\in \mathcal{E}_{I}$ such that $F\subseteq G$. So $F\subseteq G\subseteq E$%
. Since $E,F\in \mathcal{E}_{I}$, by (E1), $F=E$ and then $E=G$. Therefore $%
\mathcal{E}_{I}\subseteq \{\bigcup_{i\in K\cup L}v\Lambda ^{e_{i}}:v\in
\Lambda ^{0}\}$.

To show $\{\bigcup_{i\in K\cup L}v\Lambda ^{e_{i}}:v\in \Lambda
^{0}\}\subseteq \mathcal{E}_{I}$, take $v\in \Lambda ^{0}$ and write $%
E:=\bigcup_{i\in K\cup L}v\Lambda ^{e_{i}}$. We show $E\in \mathcal{E}_{I}$.
By \eqref{equ-ideals-Toeplitz-belong-to-I}, $\prod_{e\in
E}(t_{v}-t_{e}t_{e}^{\ast })\in I$ and then by Theorem \ref%
{gauge-invariant-ideals-Toeplitz-algebra}(c), there exists $F\in v\mathcal{E}%
_{I}$ such that $F\subseteq E$. Since $\mathcal{E}_{I}\subseteq
\{\bigcup_{i\in K\cup L}v\Lambda ^{e_{i}}:v\in \Lambda ^{0}\}$, we have $%
F=\bigcup_{i\in K\cup L}v\Lambda ^{e_{i}}$. Hence $E=F\in \mathcal{E}_{I}$.
Therefore
\begin{equation}
\mathcal{E}_{I}=\Big\{\bigcup_{i\in K\cup L}v\Lambda ^{e_{i}}:v\in \Lambda
^{0}\Big\}\text{.}  \label{equ-ideals-Toeplitz-E_I}
\end{equation}

Since $H_{I}=\emptyset $ and $\mathcal{E}_{I}=\{\bigcup_{i\in K\cup
L}v\Lambda ^{e_{i}}:v\in \Lambda ^{0}\}$ (\eqref{equ-ideals-Toeplitz-H_I}
and \eqref{equ-ideals-Toeplitz-E_I}), by Theorem \ref%
{gauge-invariant-ideals-Toeplitz-algebra}(d),
\begin{equation*}
\Big<\Big\{\prod_{i\in K\cup L}\prod_{e\in v\Lambda
^{e_{i}}}(t_{v}-t_{e}t_{e}^{\ast }):v\in \Lambda ^{0}\Big\}\Big>%
=I=\left\langle A\right\rangle \cap \left\langle B\right\rangle \text{.}
\end{equation*}%
On the other hand, we have%
\begin{align*}
AB& =\{ ab:a\in A,b\in B\} \\
& =\Big\{\Big(\prod_{i\in K}\prod_{e\in v\Lambda
^{e_{i}}}(t_{v}-t_{e}t_{e}^{\ast })\Big)\Big(\prod_{j\in L}\prod_{e\in
v\Lambda ^{e_{j}}}(t_{v}-t_{e}t_{e}^{\ast })\Big):v\in \Lambda ^{0}\Big\}
\end{align*}%
and then $\left\langle A\right\rangle \cap \left\langle B\right\rangle
=\langle AB\rangle $. Using a similar argument, $\left\langle A\right\rangle
\cap \left\langle B\right\rangle =\langle BA\rangle $.
\end{proof}

\begin{proof}[Proof of Proposition \protect\ref{diagram-Toeplitz-algebra}]
The second row and the second column are exact. On the other hand, using
the third isomorphism theorem, we get the exactness of the third row and the
third column. Next we show that the first row and the first column are
exact. Lemma \ref{ideals-in-Toeplitz-algebra} tells $\left\langle
A\right\rangle +\left\langle B\right\rangle =\left\langle A\cup
B\right\rangle $ and $\left\langle A\right\rangle \cap \left\langle
B\right\rangle =\left\langle AB\right\rangle =\left\langle BA\right\rangle $%
. Then by the second isomorphism theorem,%
\begin{equation*}
\frac{\left\langle A\right\rangle }{\left\langle AB\right\rangle }=\frac{%
\left\langle A\right\rangle }{\left\langle A\right\rangle \cap \left\langle
B\right\rangle }\cong \frac{\left\langle A\right\rangle +\left\langle
B\right\rangle }{\left\langle B\right\rangle }=\frac{\left\langle A\cup
B\right\rangle }{\left\langle B\right\rangle }\text{,}
\end{equation*}%
\begin{equation*}
\frac{\left\langle B\right\rangle }{\left\langle BA\right\rangle }=\frac{%
\left\langle B\right\rangle }{\left\langle A\right\rangle \cap \left\langle
B\right\rangle }\cong \frac{\left\langle A\right\rangle +\left\langle
B\right\rangle }{\left\langle A\right\rangle }=\frac{\left\langle A\cup
B\right\rangle }{\left\langle A\right\rangle }\text{.}
\end{equation*}
\end{proof}

\end{document}